\documentclass[final,hidelinks,onefignum,onetabnum,reqno,nohypdvips]{siamart251216}

\usepackage{mathtools}
\usepackage{bm}
\usepackage{enumitem}
\usepackage{amsfonts}
\usepackage{graphicx}
\usepackage{epstopdf}
\usepackage{euscript}  
\usepackage{algorithmic}
\usepackage{amsopn}
\ifpdf
  \DeclareGraphicsExtensions{.eps,.pdf,.png,.jpg}
\else
  \DeclareGraphicsExtensions{.eps}
\fi

\DeclareMathOperator{\E}{\mathbb{E}}
\DeclareMathOperator{\Var}{Var}
\DeclareMathOperator{\Cov}{Cov}
\DeclareMathOperator{\Cor}{Corr}
\DeclareMathOperator{\trace}{tr}

\theoremstyle{definition}

\newcommand{\boldheading}[1]{\textbf{{#1}}.}
\newcommand{\ZZ}[1]{Z({#1})}
\newcommand{\Zx}{\ZZ{x}}
\newcommand{\Zy}{\ZZ{y}}
\newcommand{\Zxr}[1]{Z^{[{#1}]}(x)}
\newcommand{\dZr}[1]{\vec{Z}^{[{#1}]}}
\newcommand{\Nystrom}{Nystr\"{o}m's}
\newcommand{\KL}{Karhunen--Lo\`{e}ve}

\newcommand{\dx}{\mathrm{d}x}
\newcommand{\dy}{\mathrm{d}y}
\newcommand{\domega}{\mathrm{d}\omega}
\newcommand{\R}{\mathbb{R}}
\newcommand{\N}{\mathbb{N}}

\newcommand{\eps}{\varepsilon}

\newcommand{\F}{\mathcal{F}}

\renewcommand{\H}{\mathcal{H}}
\renewcommand{\P}{\mathbb{P}}
\newcommand{\B}{\mathcal{B}}
\newcommand{\mat}[1]{\mathbf{{#1}}}
\newcommand{\tran}{\mathsf{T}}
\renewcommand{\vec}[1]{{\mathchoice
                     {\mbox{\boldmath$\displaystyle{#1}$}}
                     {\mbox{\boldmath$\textstyle{#1}$}}
                     {\mbox{\boldmath$\scriptstyle{#1}$}}
                     {\mbox{\boldmath$\scriptscriptstyle{#1}$}}}}
\newcommand{\eip}[2]{\left\langle {#1}, {#2} \right\rangle_2}
\newcommand{\ip}[2]{\left\langle {#1}, {#2} \right\rangle}
\newcommand{\ipOmega}[2]{\left\langle {#1}, {#2} \right\rangle_{L^2(\Omega)}}
\newcommand{\normOmega}[1]{\|{#1}\|_{L^2(\Omega)}}

\newsiamremark{example}{Example}
\newsiamremark{remark}{Remark}
\newsiamremark{hypothesis}{Hypothesis}
\crefname{hypothesis}{Hypothesis}{Hypotheses}
\newsiamthm{claim}{Claim}
\newsiamremark{fact}{Fact}
\crefname{fact}{Fact}{Facts}
\newcommand{\D}{\EuScript{D}}

\newcommand{\defeq}{\vcentcolon=}
\newcommand{\C}{\mathrm{C}}
\newcommand{\A}{\mathrm{A}}
\newcommand{\K}{\mathrm{K}}

\newcommand{\AU}{\mathrm{A}_{\scriptscriptstyle{U}}}

\newcommand{\PU}{\mathrm{P}_{\!\scriptscriptstyle{U}}}

\newcommand{\Pb}{\bm{\mathcal{P}}}
\newcommand{\Proj}[2]{\Pb_{\!\scriptscriptstyle{#1}}^{{#2}}} 

\headers{A Primer on the KLE}{A.~Alexanderian}
\title{A Primer on the Karhunen--Lo\`{e}ve Expansion}
\author{Alen Alexanderian\thanks{North Carolina State University (\email{alexanderian@ncsu.edu}).}}

\allowdisplaybreaks

\begin{document}

\maketitle

\begin{abstract} 
This article provides a primer on the spectral representation of random fields
via the \KL{} expansion (KLE).  The goal is to bridge the gap between the
theoretical foundations of the KLE and its application in computational modeling
under uncertainty.  We detail how tools from operator theory and probability
are combined to analyze the convergence and optimality of the KLE. We  also emphasize the
associated computational and mathematical modeling considerations.
\end{abstract}

\begin{keywords}
Random field, \KL{} expansion, 
Gaussian Process, Uncertainty quantification
\end{keywords}

\begin{MSCcodes}
65C20, 
60G60, 
65R20, 
35R60 
\end{MSCcodes}

\section{Introduction}
Uncertainty quantification (UQ) in physical systems often involves modeling
spatially varying uncertain parameters as random fields~\cite{Smith24}.  These field quantities
may represent source terms, boundary conditions, or heterogeneous material
properties in partial differential equations (PDEs) governing physical
processes.  To illustrate, we depict a sample realization of a random field
model for the log-permeability field of a heterogeneous medium in
Figure~\ref{fig:perm}.
\begin{figure}[ht]\centering
\includegraphics[width=.45\textwidth]{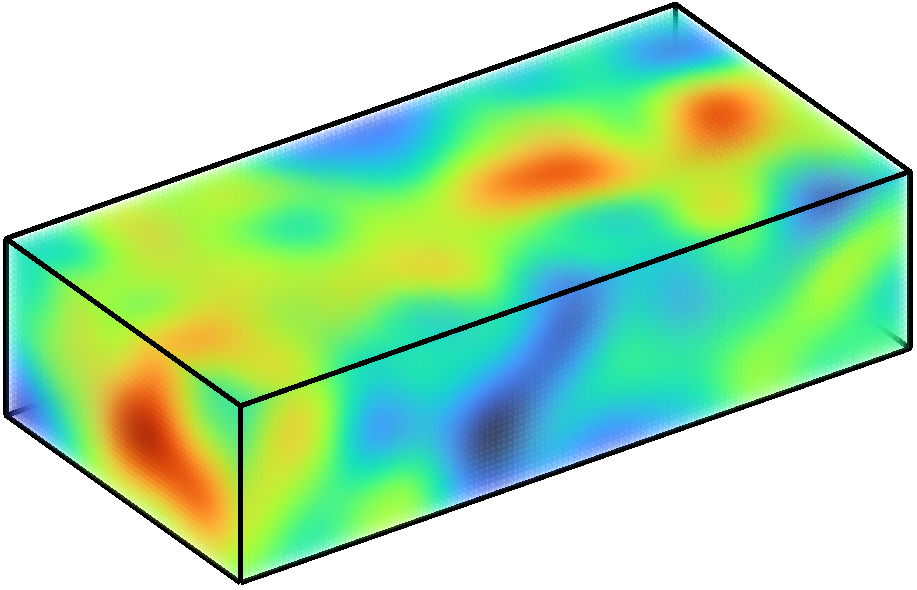}
\caption{A realization of a random field model for the log-permeability field 
of a heterogeneous medium. The bright region in the center indicates 
a high permeability channel.}
\label{fig:perm}
\end{figure}

The \KL{} expansion (KLE)~\cite{Karhunen47,Loeve77} is a fundamental tool in
UQ: it provides an efficient spectral representation of 
random fields.  The KLE is also closely related to Principal Component Analysis 
(PCA)~\cite{Jolliffe02} and Proper Orthogonal Decomposition
(POD)~\cite{BerkoozHolmesLumley93}, which are well-known dimension reduction
approaches in statistical learning and computational fluid dynamics,
respectively. The differences between these techniques are primarily in
perspective and applications.

The analysis of the KLE draws upon tools from functional analysis and probability
theory; the associated numerical implementations require methods from numerical
analysis and careful attention to modeling considerations.  Synthesizing the
requisite concepts from these areas involves studying multiple
sources that target readers with different backgrounds and often employ
different systems of notation.  The goal of this article, which evolved from the
note~\cite{Alexanderian15} I wrote over ten years ago, is to bridge the
gap between theory and applications of the KLE in an accessible manner.  

We provide a rigorous account of the underlying theory, followed by a focused
discussion of modeling and computations with random fields.  To provide an
accessible presentation, some of the key theoretical results are first discussed
in a finite-dimensional setting before considering infinite-dimensional Hilbert
space formulations.  Also, in some cases, to build intuition, we follow a
constructive approach where the derivations precede the precise statements of
the results.  

\boldheading{Audience}
This article is primarily intended for graduate students who have a strong 
background in linear algebra, are familiar with elements of Hilbert space theory
and probability, and have some knowledge of core concepts from numerical
analysis.  The topics considered herein can be integrated into courses on
uncertainty quantification, Bayesian inverse problems, or stochastic PDEs.
Beyond its pedagogical goals, this article also aims to serve as a 
reference on the KLE for researchers in uncertainty quantification, 
spatial statistics, and machine learning.

\boldheading{Article overview}
In Section~\ref{sec:motivation}, we consider a spectral representation for
random \emph{vectors}, which is called the \emph{discrete} KLE by some authors; 
see, e.g.,~\cite[Section 9.1]{Fukunaga13}.  The discrete KLE of a random vector
$\vec{X}$ is also known as the PCA of $\vec{X}$; see, e.g.,~\cite[Section
5.7]{Soize17}.  We adopt the term discrete KLE to remain consistent with the
continuous case.  The discussion in Section~\ref{sec:motivation} provides a
blueprint for our study of random \emph{fields}. Namely, by restricting the
mathematical prerequisites to linear algebra and basic probability, that section
aims to provide the core insights behind spectral analysis of random fields by
their KLE.

While the discrete KLE of a random vector relies on the spectral decomposition
of its covariance \emph{matrix}, the KLE of a random field is tied to the
spectral representation of its covariance \emph{operator}. The latter 
involves an interesting interplay between concepts from functional analysis and
random field theory.  A brief outline of the requisite background concepts is
provided in Section~\ref{sec:background}.  Subsequently, in
Section~\ref{sec:analysis}, we analyze the convergence and optimality of the
KLE. We also discuss the special case of Gaussian random fields in that section. 

In Section~\ref{sec:discretization}, we discuss discretization of
the KLE, which is needed to facilitate numerical computations
with random fields. Our discussion in that section  
focuses primarily on \Nystrom{} method.
Section~\ref{sec:numerics} provides a bridge to practical computations.  
There, we focus on Gaussian random fields and consider various modeling,
mathematical, and computational considerations, along with 
illustrative computational experiments.  
Concluding remarks are provided in the epilogue (Section~\ref{sec:conc}), where
we mention references for further reading and provide perspectives on
computations with non-Gaussian random fields and goal-oriented dimension
reduction.

\boldheading{Notation}
We denote the Euclidean inner product of  
vectors $\vec{x}$ and $\vec{y}$ in $\R^n$ by 
$\eip{\vec x}{\vec y} \defeq \vec{y}^\tran\vec{x}$.
The Euclidean norm 
of $\vec{x} \in \R^n$ is denoted by $\| \vec{x} \|_2 \defeq \eip{\vec x}{\vec x}^{1/2}$.
For a real-valued random variable $X$, we denote its expectation and variance by 
$\E[X]$ and $\Var[X]$, respectively. We use the standard notation $\delta_{ij}$
for the Kronecker delta: $\delta_{ij}$ equals one if $i = j$, and 
zero otherwise.
Throughout the article, $\D$ denotes a spatial region on which we define random
fields.  We primarily consider the case where 
$\D$ is a closed and bounded interval  $[a, b] \subset
\mathbb{R}$.  This is done for simplicity of the presentation---the entire
analysis that follows extends naturally to random fields on compact subsets of
$\R^d$ with $d > 1$. 
Herein, $L^2(\D)$ is the space of 
(equivalence classes of) real-valued square integrable functions on $\D$. 
This space is equipped with the inner product 
$\ip{u}{v} \defeq \int_\D u(x) v(x) \, \dx$
and norm $\| \cdot \| = \ip{\cdot}{\cdot}^{1/2}$.

\section{Motivation: the discrete KLE}\label{sec:motivation}
Consider a vector-valued random variable $\vec{Z}:\Omega \to \R^n$, where
$\Omega$ is a sample space.  We assume $\E[\|\vec Z\|_2^2] < \infty$. Also, for
simplicity, suppose $\vec{Z}$ has zero mean; i.e., $\E[\vec{Z}] = \vec{0}$.  
Subsequently, the covariance matrix of $\vec{Z}$ is given by
\begin{equation}\label{equ:cov_matrix}
C_{ij} = \E[Z_i Z_j] - \E[Z_i]\E[Z_j] = \E[Z_i Z_j], \quad i, j \in \{1, \ldots, n\}.
\end{equation}
Since $\E[\| \vec{Z}\|_2^2] < \infty$, by the Cauchy--Schwarz
inequality, $\E[Z_i Z_j] < \infty$ for all $i, j$.
Also, $\mat{C}$ is symmetric and positive semidefinite. 
To see the latter, note that 
$\mat{C} =
\E[\vec{Z}\vec{Z}^\tran]$.

\subsection{A spectral representation}\label{sec:fd-study}
Consider the spectral decomposition of $\mat{C}$, 
\begin{equation}\label{equ:spectral}
\mat{C} = \sum_{i=1}^n \lambda_i \vec{v}_i \vec{v}_i^\tran.
\end{equation}
Here, 
$\{\vec{v}\}_{i=1}^n$ is an orthonormal basis of eigenvectors of
$\mat{C}$ and 
$\{\lambda_i\}_{i=1}^n$ are the corresponding eigenvalues 
listed in descending order,
$\lambda_1 \geq \lambda_2 \geq \cdots \geq \lambda_n \geq 0$.

We can represent $\vec{Z}(\omega)$ in the 
eigenbasis of $\mat{C}$. Namely,
for $\omega \in \Omega$, 
\begin{equation}\label{equ:KLE_FD}
\vec{Z}(\omega) = \sum_{i=1}^n z_i(\omega) \vec{v}_i,
\quad \text{where}\quad z_i(\omega) = \eip{\vec Z(\omega)}{\vec{v}_i}. 
\end{equation}
This spectral representation of $\vec{Z}$ is the discrete KLE.
Note that 
\[
\E[z_i] = 
\E[\eip{\vec{Z}}{\vec{v}_i}] = 
\E[\vec{v}_i^\tran \vec{Z}] =
\vec{v}_i^\tran \E[\vec{Z}] = 0,
\]
for each $i \in \{1, \ldots n\}$.
Moreover, for $i, j \in \{1, \ldots, n\}$,
\[
\begin{aligned}
\E[z_i z_j] 
  = \E[\eip{\vec{Z}}{\vec{v}_i} \eip{\vec{Z}}{\vec{v}_j}]
  &= \E[\vec{v}_i^\tran \vec{Z}\vec{Z}^\tran \vec{v}_j]\\
  &= \vec{v}_i^\tran \E[\vec{Z}\vec{Z}^\tran] \vec{v}_j
  = \vec{v}_i^\tran \mat{C} \vec{v}_j = \lambda_i \delta_{ij}.
\end{aligned}
\]
The following result summarizes the preceding discussion: 
\begin{proposition}\label{prp:KLE_FD}
Let $\vec{Z}$ be a mean zero random vector with $\E[\| \vec{Z}\|_2^2 ] < \infty$.
Then, $\vec{Z}$ can be represented according to~\eqref{equ:KLE_FD}, where 
the coefficients $z_i$ are mean zero random variables 
that satisfy 
$\E[z_i z_j] = \lambda_i \delta_{ij}$, for $i, j \in \{1, \ldots, n\}$.
\end{proposition}

By Proposition~\ref{prp:KLE_FD}, we have that $\Var[z_i] = \E[z_i^2] = \lambda_i$ for each
$i \in \{1, \ldots, n\}$. Observe that 
if $\lambda_i = 0$, then $\Var[z_i] = \E[z_i] = 0$. Thus, the  
coefficients corresponding to zero eigenvalues vanish almost surely. 
To avoid such degenerate cases, in what follows,  we assume
$\vec{Z}$ is such that $\mat{C}$ is positive definite. 

Note that~\eqref{equ:KLE_FD} provides a decomposition of $\vec{Z}$ 
where the expansion coefficients are mean zero and uncorrelated random variables. 
However, the coefficients have different variances. In practice, it is common 
to consider a spectral representation in terms of standardized (unit-variance) 
random coefficients. Namely, we consider 
\begin{equation}\label{equ:KLE_FD_normalized}
\vec{Z}(\omega) = \sum_{i=1}^n \sqrt{\lambda_i} \xi_i(\omega) \vec{v}_i, 
\quad \text{where} \quad
\xi_i(\omega) = z_i(\omega) / \sqrt{\lambda_i}.
\end{equation}
This provides a spectral representation of $\vec{Z}$ in terms of standardized
uncorrelated coefficients $\{\xi_i\}_{i=1}^n$ and exposes the eigenvalues 
explicitly.

The representation~\eqref{equ:KLE_FD_normalized} is useful 
from a computational perspective. Namely, if $\vec{Z}$ is a high-dimensional 
random variable and $\mat{C}$ has rapidly decaying eigenvalues, we can 
approximate $\vec{Z}$ by a truncated expansion 
\[
\vec{Z}(\omega) \approx \dZr{r}(\omega) \defeq \sum_{i=1}^r \sqrt{\lambda_i} \xi_i(\omega) \vec{v}_i,
\]
with $r \ll n$.  This is an important and widely used dimension reduction idea.

We can also quantify the approximation error due to truncation:
\begin{equation}\label{equ:err_KLE}
\E\Big[\| \vec{Z} - \dZr{r}\|_2^2\Big] = \E\Big[\big\|\!\!\sum_{i={r+1}}^n z_i \vec{v}_i\big\|_2^2\Big]
= \sum_{i=r+1}^n \lambda_i. 
\end{equation} 
Subsequently, we can express the relative mean-square error as 
\[
    \frac{\E\Big[\| \vec{Z} - \dZr{r}\|_2^2\Big]}{\E\Big[\| \vec{Z}\|_2^2\Big]}
    = 
    \frac{\sum_{i=r+1}^n \lambda_i}{\sum_{i=1}^n \lambda_i} 
    = 1 -  \frac{\sum_{i=1}^r \lambda_i}{\sum_{i=1}^n \lambda_i}.
\]
This suggests a criterion for choosing a truncation level: pick $r$ such that 
the ratio
\[
\rho(r) \defeq \frac{\sum_{i=1}^r \lambda_i}{\sum_{i=1}^n \lambda_i}
\] 
is larger than a prespecified 
threshold. 

Note also that the total variance of $\vec{Z}$, given by the sum of variances 
of its components, satisfies 
$\sum_{i=1}^n \Var[Z_i] = \trace(\mat{C}) = \sum_{i=1}^n \lambda_i$. Therefore, 
$\rho(r) = (\sum_{i=1}^r \lambda_i)/\trace(\mat{C})$.
Furthermore, $\sum_{i=1}^r \lambda_i$ is the total variance of $\dZr{r}$.
Thus, $\rho(r)$ may also be interpreted as the percentage of total variance 
captured by the truncated KLE with $r$ terms.

\subsection{Optimality properties of the discrete KLE}\label{sec:optimality}
In principle, one could represent $\vec{Z}$ 
in any orthonormal basis. 
However, there is a notion of optimality 
attached to the basis of eigenvectors of $\mat{C}$. To make matters concrete, 
let $U = \{\vec u_i\}_{i=1}^n$ be an arbitrary 
orthonormal basis of $\R^n$. We can expand
$\vec{Z}$ as 
\begin{equation}\label{equ:Fourier_expansion}
\vec{Z}(\omega) = \sum_{i=1}^n \alpha_i(\omega) \vec{u}_i,
\quad \text{where}\quad \alpha_i(\omega) = \eip{\vec Z(\omega)}{\vec{u}_i}. 
\end{equation}
Arguments similar to  
those leading to Proposition~\ref{prp:KLE_FD} yield 
$\E[\alpha_i] = 0$, 
for each $i$, and 
\[
\E[ \alpha_i \alpha_j] =  \eip{\mat{C} \vec{u}_i}{\vec{u}_j}, 
\quad i, j \in \{1, \ldots, n\}. 
\]
The latter reveals the first shortcoming of representing $\vec{Z}$ in an
arbitrary orthonormal basis: we no longer have a decomposition with 
mutually uncorrelated coefficients. 

We next ask the following question: can we choose the basis $U = \{
\vec{u}_i\}_{i=1}^n$ such that the approximation error $\eps(U,r) \defeq \E[ \|
\vec{Z} - \sum_{i=1}^r \alpha_i \vec{u}_i \|_2^2]$ is smaller than the error
obtained in~\eqref{equ:err_KLE}? As seen shortly, the answer is no. 

Observe that
\begin{multline}\label{equ:err_Fourier_generic}
\eps(U,r) =  \E\Big[\big\|\sum_{i={r+1}}^n \alpha_i \vec{u}_i\big\|_2^2\Big] 
= \sum_{i=r+1}^n \E[ \alpha_i^2] 
= \sum_{i=r+1}^n \eip{\mat{C} \vec{u}_i}{\vec{u}_i}\\
= \sum_{i=1}^n \eip{\mat{C} \vec{u}_i}{\vec{u}_i} - \sum_{i=1}^r \eip{\mat{C} \vec{u}_i}{\vec{u}_i} 
= \trace(\mat{C}) - \sum_{i=1}^r \eip{\mat{C} \vec{u}_i}{\vec{u}_i}.
\end{multline}
Next, let 
$V = \{\vec{v}_i\}_{i=1}^n$ be the eigenbasis of $\C$, ordered as
in~\eqref{equ:spectral}. That is, the eigenpairs are ordered such that 
$\lambda_1 \geq \lambda_2 \geq \cdots \geq \lambda_n > 0$. 
We have 
\[
\eps(V, r) = \sum_{i=r+1}^n \lambda_i = \trace(\mat{C}) - \sum_{i=1}^r \lambda_i. 
\]
Therefore,
\[
\eps(U, r) - \eps(V, r) = \sum_{i=1}^r \lambda_i - \sum_{i=1}^r \eip{\mat{C} \vec{u}_i}{\vec{u}_i}
=  
\sum_{i=1}^r \lambda_i - \trace(\mat{U}_r^\tran \mat{C} \mat{U}_r),
\]
where $\mat{U}_r = [ \vec{u}_1 \; \vec{u}_2 \; \cdots \; \vec{u}_r]\in \R^{n \times r}$.
Note that $\mat{U}_r^\tran \mat{U}_r = \mat{I}$. 

We now invoke 
the Ky Fan Maximum Principle (see~\cite[Page 309]{Dym13} or~\cite[Page 35]{Bhatia97}), 
which implies
\begin{equation}\label{equ:KyFanQ}
\trace(\mat{Q}^\tran \mat{C} \mat{Q}) \leq \sum_{i=1}^r \lambda_i,
\quad \text{for all $\mat{Q} \in \R^{n \times r}$ with } \mat{Q}^\tran \mat{Q} = \mat{I}.
\end{equation}
Using this result yields,
$\eps(U, r) - \eps(V, r) = \sum_{i=1}^r \lambda_i - \trace(\mat{U}_r^\tran \mat{C} \mat{U}_r) \geq 0$.
That is, 
\[
   \eps(U, r) \geq \eps(V, r).
\]
This reveals a well-known~\cite{Ghanem,Fukunaga13} optimality property of the 
representation~\eqref{equ:KLE_FD}: expanding $\vec{Z}$ in the basis of
eigenvectors of $\mat{C}$ yields minimal mean-square error when using truncated
expansions. We state a precise statement of this result next. 

For an orthonormal basis 
$U = \{\vec{u}_1, \ldots, \vec{u}_n\} \subset \R^n$
and $r \in \{1, \ldots, n\}$,
we let 
\[
\Proj{U}{r} \defeq \sum_{i=1}^r \vec{u}_i \vec{u}_i^\tran
\]
denote the orthogonal projector onto the span of $\{\vec{u}_1, \ldots, \vec{u}_r\}$.
The preceding constructive argument proves the following well-known result:
\begin{proposition}\label{prp:optimality}
Let $\vec{Z}$ be a mean zero random vector with $\E[\| \vec{Z}\|_2^2 ] < \infty$ 
and a positive definite covariance matrix $\mat{C}$. Let 
$\{(\lambda_i, \vec{v}_i)\}_{i=1}^n$ be the eigenpairs of $\mat{C}$, 
with the eigenvalues listed in descending order, $\lambda_1 \geq \lambda_2 \geq \cdots 
\geq \lambda_n > 0$. Denote $V = \{\vec{v}_1, \ldots, \vec{v}_n\}$.
Then, for every orthonormal basis $U = \{\vec{u}_1, \vec{u}_2, \ldots, \vec{u}_n\}$ 
of $\R^n$,
\[
\E[ \| \vec{Z} - \Proj{V}{r}\vec{Z} \|_2^2] \leq 
\E[ \| \vec{Z} - \Proj{U}{r}\vec{Z} \|_2^2], \quad \text{for each } r \in \{1, \ldots, n\}.
\]
\end{proposition}

There are several approaches to proving
Proposition~\ref{prp:optimality}.  Some authors justify the result directly
using a variational argument involving a Lagrange multiplier approach; see,
e.g.,~\cite[Section 9.1]{Fukunaga13} or~\cite[Section 2.3.2]{Ghanem}.  Others
rely on deeper results regarding the eigenvalues of symmetric positive
semidefinite matrices; see, e.g.,~\cite{Keating83,Dur98}.  Our constructive
proof of Proposition~\ref{prp:optimality} is closely related to the one
in~\cite{Keating83} that relies on the
Poincar\'{e} Separation Theorem~\cite[Page 64]{Rao73}.  
In our context, however, invoking the Ky
Fan Maximum Principle was a natural choice.  Ky Fan's result, which extends naturally to an
infinite-dimensional Hilbert space setting, is used in
Section~\ref{sec:analysis} to prove the optimality of the KLE for 
random fields.  

The optimality property in Proposition~\ref{prp:optimality} 
can also be understood in terms of the total
variance captured by a spectral representation of $\vec{Z}$.
Specifically, for a random vector $\vec Z$ with
$\E[\| \vec Z\|_2^2] < \infty$ and a rank-$r$ orthogonal projector $\Pb$, we
have 
\begin{equation}\label{equ:var_decomp}
\underbracket[0.5pt][0.7ex]{\mathbb{E}[\|\vec Z\|_2^2]}_{\text{total variance}} = 
   \underbracket[0.5pt][0.7ex]{\mathbb{E}[\|\Pb\vec Z\|_2^2]}_{\text{captured variance}} + 
   \underbracket[0.5pt][0.7ex]{\mathbb{E}[\|(\mat{I} - \Pb)\vec Z \|_2^2]}_{\text{variance of the error}}.
\end{equation}
In the case of a truncated discrete KLE with $r$ terms, $\Pb \vec Z$ is the orthogonal
projection onto the span of the $r$ leading eigenvectors $\{\vec{v}_i\}_{i=1}^r$
of the covariance matrix $\mat{C}$. Consequently, in view of
\eqref{equ:var_decomp}, the truncated discrete KLE, $\dZr{r} = \sum_{i=1}^r
\eip{\vec{Z}}{\vec{v}_i}\vec{v}_i$, \emph{maximizes} the total variance captured
by the orthogonal projection $\Pb \vec{Z}$ over all rank-$r$ orthogonal projectors $\Pb$. 
Equivalently, the truncated KLE \emph{minimizes} the total variance of the
error among all such projections. 

\section{Analysis and probability background}\label{sec:background}
In this section, we provide a concise summary of the requisite background concepts from
functional analysis (Section~\ref{sec:FA}) and probability theory 
(Section~\ref{sec:RFs}).

\subsection{Functional analysis essentials}\label{sec:FA}
Let $\H$ be a real Hilbert space equipped with an inner product
$\ip{\cdot}{\cdot}$ and norm $\| \cdot \| = \ip{\cdot}{\cdot}^{1/2}$.  
A linear operator $\A:\H \to \H$ is said to be
bounded if there exists a real constant $M$ with 
\[
    \| \A x \| \leq M \| x \|, \quad \text{for all } x \in \H.
\]

Recall that a set $E \subset \H$ is compact if every
sequence $\{ z_k \}_{k \in \N}$ in $E$ has a subsequence that converges in $E$.
A linear operator $\K \colon \H \to \H$ is said to 
be compact if for each bounded set $E \subset \H$, 
the closure of its image $\overline{\K(E)}$
is a compact set. Equivalently,  
$\K$ is compact if 
for every bounded sequence $\{x_k\}_{k \in \N} \subset \H$, 
the sequence $\{\K x_k\}_{k \in \N}$ has a 
convergent subsequence.

Much theory exists
regarding compact operators.  
In the present discussion, we consider the Hilbert
space $L^2(\D)$ and are concerned
primarily with a specific type of compact operators on this space.
Namely, we consider \emph{integral operators} of the form 
\begin{equation}\label{equ:HS}
[\K u](x) = \int_\D k(x, y) u(y) \, \dy, \quad u \in L^2(\D), \, x\in \D,
\end{equation}
where $k$ is continuous on $\D \times \D$ and is symmetric; i.e., 
$k(x, y) = k(y, x)$ for all $x, y \in \D$. In this context, $k$ is called the \emph{kernel}
of the integral operator.
Our motivation behind considering integral operators is their role in the study
of random fields. As discussed in Section~\ref{sec:RFs}, the covariance operator
of a random field is an integral operator with the associated covariance
function as the kernel.

The operator $\K$ is a standard example of a compact linear operator 
and is discussed in many functional analysis textbooks; see, e.g.,~\cite[Pages 384--385]{naylorsell}.
Due to symmetry of the kernel, this operator is also \emph{self-adjoint}:
\[
  \ip{\K u}{v} = \ip{u}{\K v}, \quad \text{for all } u, v \in L^2(\D).
\]
We record the following technical result regarding operators of this type. 
\begin{lemma} \label{lem:compactop}
Let $k:\D \times \D \to \R$ be a continuous and symmetric function. Then, 
the operator
$\K : L^2(\D) \to L^2(\D)$ defined in~\eqref{equ:HS} 
is a compact self-adjoint operator.
\end{lemma}

In what follows, we assume $k$ in Lemma~\ref{lem:compactop} is such that 
the corresponding integral operator $\K$ is positive. That is, 
$\ip{\K u}{u} \geq 0$, for all $u \in L^2(\D)$. 
For clarity,
we refer to such kernels as \emph{admissible} kernels. This is made precise 
in the following definition.
\begin{definition}
We call a function $k: \D \times \D \to \R$ an admissible kernel 
on $\D$ if it is continuous and symmetric, and the integral operator 
\eqref{equ:HS} is positive. 
\end{definition}

\boldheading{Spectral representation of compact self-adjoint operators}
A compact self-adjoint operator on an infinite dimensional 
real Hilbert space provides a Hilbert space analogue of a real symmetric matrix. 
In particular, such operators admit a spectral representation.
Let $\K$ be a compact self-adjoint operator on $\H$. By the Spectral Theorem for compact 
self-adjoint operators, 
$\K$ has real eigenvalues $\{\lambda_i\}_{i=1}^\infty$ with corresponding eigenvectors 
$\{v_i\}_{i=1}^\infty$, 
\[
\K v_i = \lambda_i v_i, \quad i \in \N,
\]
and
the eigenvectors form an orthonormal basis of $\H$. Furthermore, for every 
$u \in \H$,
\[
\K u = \sum_{i=1}^\infty \lambda_i \ip{u}{v_i} v_i,
\]
where the series converges in norm; that is, $\| \K u -
\sum_{i=1}^n \lambda_i \ip{u}{v_i} v_i\| \to 0$ as $n\to\infty$. For details on 
spectral theory of compact self-adjoint operators see, e.g.,
\cite[Chapter II]{Knapp05}
or~\cite[Chapter 6]{naylorsell}.
For a positive compact self-adjoint operator, 
the eigenvalues are non-negative and we may
order them as 
$\lambda_1 \geq \lambda_2 \geq \cdots \geq 0$.

The following classical result is the Hilbert space version of the Ky Fan
Maximum Principle considered in Section~\ref{sec:optimality}.  For completeness,
we present a proof of this result in Appendix~\ref{sec:KyFanProof}.  The proof
uses the Courant--Fischer min-max theorem, which provides a variational
characterization of eigenvalues of a compact positive selfadjoint operator; see,
e.g., \cite[Theorem 4.2.7]{HsingEubank15} or~\cite[Page 247]{HirschLacombe99}.
Such arguments are standard in matrix analysis and operator
theory.

\begin{proposition}\label{prp:KyFanHilbert}
Let $\H$ be a real separable Hilbert space with inner product $\ip{\cdot}{\cdot}$. And let 
$\A : \H \to \H$ be a positive compact self-adjoint operator with eigenvalues
\[
\lambda_1 \geq \lambda_2 \geq \cdots \geq 0,
\]
counting multiplicities. Then, for every orthonormal set 
$\{u_1, \ldots, u_r\} \subset \H$,
\begin{equation}\label{equ:KyFan}
\sum_{i=1}^r \ip{\A u_i}{u_i}
\leq 
\sum_{i=1}^r \lambda_i.
\end{equation}
\end{proposition}
\begin{proof}
See Appendix~\ref{sec:KyFanProof}.
\end{proof}
Observe that 
in~\eqref{equ:KyFan} equality is attained by letting  
$\{u_1,\ldots,u_r\}$ be an orthonormal set of eigenvectors 
corresponding to the leading eigenvalues $\lambda_1,\ldots,\lambda_r$.

\boldheading{Back to integral operators: a spectral representation for the kernel} \label{sec:mercer}
The Spectral Theorem for compact self-adjoint operators is a fundamental tool in
spectral analysis of random fields; it facilitates an infinite-dimensional
analogue of~\eqref{equ:KLE_FD}. Analyzing the convergence of the resulting
infinite series requires invoking another key result:  we need \emph{Mercer's}
Theorem, which provides a spectral representation for the kernel of an integral
operator using the eigenpairs of the operator.

\begin{theorem}[Mercer] \label{thm:mercer}
Let $k: \D \times \D \to \R$ be an admissible kernel on $\D$. 
If $\{(\lambda_i, v_i)\}_{i=1}^\infty$ are the eigenpairs of the integral 
operator $\K$ defined in \eqref{equ:HS}, then
\begin{equation} \label{equ:kernel-dec}
k(x, y) = \sum_{i=1}^\infty \lambda_i v_i(x) v_i(y),
\end{equation}
where the series converges absolutely and uniformly in $\D \times \D$.
\end{theorem}

For further details on Mercer's theorem, see, e.g.,~\cite[Chapter V]{GohbergGoldberg04}.

\boldheading{Trace class operators}
Let $\H$ be a real separable Hilbert space equipped with the inner product
$\ip{\cdot}{\cdot}$ and let $\{e_i\}_{i=1}^\infty$ be an orthonormal basis of
$\H$.  Assume $\A: \H \to \H$ is a bounded positive
self-adjoint operator. We say $\A$ is \emph{trace class} if 
\begin{equation}\label{equ:trace_class}
\sum_{i=1}^\infty \ip{ \A e_i}{e_i} < \infty.
\end{equation}
Note that if~\eqref{equ:trace_class} holds for one choice of the orthonormal
basis, it holds for any other orthonormal basis, and the value of the summation 
is independent of
the basis.  In this case, we denote $\trace(\A) \defeq \sum_{i=1}^\infty \ip{ \A
e_i}{e_i}$.  If $\A$ is trace class, it is also compact. Also, letting 
$\{v_i\}_{i=1}^\infty$ be the orthonormal basis of eigenvectors of $\A$, with 
corresponding (real, non-negative) eigenvalues $\{\lambda_i\}_{i=1}^\infty$, we
have $\trace(\A) = \sum_{i=1}^\infty  \ip{ \A v_i}{v_i} = \sum_{i=1}^\infty
\lambda_i$. 

For further details on trace class operators see, e.g.,~\cite{ReedSimon72}.
Such operators play an important role in the study of random fields. 
In particular, 
as seen in the next section, 
the random fields considered in this article admit trace class
covariance operators.

\subsection{Basics from probability and random fields} \label{sec:RFs}
Herein, we let $(\Omega, \F, \P)$ be a probability space, 
where $\Omega$ is a sample space, $\F$ is a $\sigma$-algebra on $\Omega$, 
and $\P$ is a probability measure. The sample space is the set of all possible 
outcomes. The $\sigma$-algebra $\F$ is a collection of subsets of $\Omega$, 
which we call \emph{events}. This collection contains $\Omega$ and is 
closed under complements and countable unions. Every $E \in \F$ is assigned 
a probability $\P(E)$. The probability measure $\P$ 
a function from $\F$ to $[0, 1]$ that satisfies the following properties: 
$\P(\Omega) = 1$ and $\P$ 
is countably additive. The latter means
\begin{equation}
\P\left(\cup_{i=1}^\infty E_i\right) = \sum_{i=1}^\infty \P(E_i),
\end{equation}
whenever $\{E_i\}_{i=1}^\infty$ is a collection of pairwise disjoint events in $\F$.
For further details on probability spaces, see, e.g.,~\cite{Williams91}.

\boldheading{Real-valued random variables} 
We equip $\R$ with the Borel $\sigma$-algebra, $\B(\R)$, which is the
smallest $\sigma$-algebra that contains all open subsets of $\R$.
A real-valued random variable on $(\Omega, \F, \P)$ is a function $X:\Omega \to \R$ for which
\[
X^{-1}(B) \in \F, \quad \text{for all } B \in \B(\R).
\]
In parlance of measure theory, a random variable is called a \emph{measurable function}.
The expectation and variance of $X$ are, respectively, denoted by
\[
   \E[X] := \int_\Omega X(\omega) \, \P(d\omega) \quad \text{and}
   \quad \Var[X] := \E[(X - \E[X])^2].
\]
Also, 
the covariance and correlation between random variables 
$X$ and $Y$ are given by 
\[
\begin{aligned}
\Cov[X, Y] &\defeq \E[(X - \E[X])(Y - \E[Y])] \quad \text{and}\\
\Cor[Y, Z] &\defeq \frac{\Cov[Y, Z]}{\Var[Y]^{1/2} \Var[Z]^{1/2}},
\end{aligned}
\]
respectively.
We denote by 
$L^2(\Omega, \F, \P)$
the Hilbert space of real-valued
square integrable random variables on $\Omega$,
$L^2(\Omega, \F, \P) =
   \{ X : \Omega \to \R : \E\{X^2\} < \infty\}$.
This space is equipped 
with the inner product, $\ipOmega{X}{Y} = \E[XY]$
and norm $\normOmega{X} = \ipOmega{X}{X}^{1/2}$.

\boldheading{Random fields}
As before, 
let $\D \subset \R$ be a closed and bounded interval. 
A random field on $\D$ is a mapping $Z:\D \times \Omega \to
\R$ such that $Z(x, \cdot)$ is measurable for every $x \in \D$. Alternatively,
we may define a random field as a family of random variables, $\Zx:\Omega
\to \R$ with $x \in \D$, and identify $Z$ with $\{\Zx\}_{x \in \D}$. We thus write
\[
Z(x, \omega) \equiv \Zx(\omega), \quad \text{for all } (x, \omega) \in \D \times \Omega.
\]
These perspectives on a random field are both useful 
and will be used in what follows. 

A random field $Z:\D \times \Omega \to \R$ is
said to be 
\begin{itemize}
\item \emph{centered} if $\E[\Zx] = 0$ for all $x \in
\D$; 
\item \emph{second-order} if $\E[\Zx^2] < \infty$ for all $x \in \D$;
\item \emph{mean-square continuous}
if for every $x \in \D$,
\[
   \displaystyle \lim_{\eps \to 0} \E[(\ZZ{x+\eps} - \Zx)^2] = 0. 
\]
\end{itemize}

In what follows, we focus on centered, second-order, mean-square continuous random
fields.  The assumption of zero mean is made for 
convenience---non-centered processes can be centered by subtracting their mean function.
On the other hand, the second-order assumption is important; it ensures a 
well-defined covariance function. To see this, consider a centered
second-order random field $Z$, with covariance function
\[
   c(x, y) = \Cov[\Zx, \Zy] = \E[\Zx \Zy], \quad x, y \in \D.
\]  
Since $Z$ is second-order, $c(x, x) = \E[\Zx^2] < \infty$. Also, by 
the Cauchy--Schwarz inequality, 
$|c(x, y)| \leq  \E[\Zx^2]^{1/2}  \E[\Zy^2]^{1/2} < \infty$. Thus, 
the second-order assumption implies that 
$c(x, y)$ is finite for every $x, y \in \D$.

We next discuss the third assumption on $Z$---mean-square continuity.  This is a
regularity assumption on the random field that has some important 
consequences. We first consider the following well-known result, which provides
a useful characterization of mean-square continuity:
\begin{lemma}\label{lem:cont}
A centered second-order random field $\{\Zx\}_{x \in \D}$ is mean-square continuous if and
only if its covariance function is continuous
on $\D \times \D$. 
\end{lemma}
\begin{proof}
See Appendix~\ref{appdx:proof_cont}.
\end{proof}

Let $Z$ be a centered, second-order, mean-square continuous random field, and
consider the following argument:
\begin{multline}\label{equ:fubini}
\E[ \| Z \|^2 ] = \int_\Omega \int_\D Z(x, \omega)^2 \, \dx\, \P(\domega) 
= \int_{\D}\int_\Omega  Z(x, \omega)^2 \, \P(\domega)\, \dx \\
= \int_\D \E[\Zx^2] \, \dx 
= \int_{\D} c(x, x)  \, \dx < \infty.
\end{multline}
In the last step, we 
have used continuity of $c$ and the fact that $\D$ is a compact set.
This is meant to show that  $Z \in L^2(\D \times \Omega)$ and 
that $Z(\cdot, \omega) \in L^2(\D)$ for almost all $\omega \in \Omega$. 
The interchange of integrals in~\eqref{equ:fubini} requires an application of 
Tonelli's theorem. The present argument is predicated on 
joint measurability of the integrand;  i.e.,
$Z$ must be measurable as a function $Z : (\D \times \Omega, \B(\D) \otimes \F) \to (\R, \B(\R))$, 
which is not a property satisfied by any arbitrary random field. 

To fill the gap in the preceding discussion, we need to utilize another implication of mean-square continuity. 
First note that mean-square continuity implies a weaker notion of 
continuity known as \emph{continuity in probability} or \emph{stochastic continuity}. 
Subsequently, we invoke a result
from probability that states a stochastically continuous 
random field $Z(x, \omega)$ admits a
measurable modification;\footnote{Random fields $Z$ and $\tilde{Z}$ are said to
be modifications of each other if $\P\big(\big\{\omega \in \Omega : Z(x, \omega)
= \tilde{Z}(x, \omega)\big\}\big) = 1$ for all $x \in \D$.} 
see, e.g.,~\cite[Proposition 3.2]{DaPratoZabczyk14}.
By passing to this
modification, joint measurability is ensured. This justifies the application of
Tonelli's theorem in the preceding discussion.  

Finally, Lemma~\ref{lem:cont} provides a bridge from theory of random fields to
operator  theory.  Specifically, by ensuring the continuity of the covariance
function, we can use Lemma~\ref{lem:cont} to conclude that the covariance
operator of a centered, second-order, mean-square continuous random field $Z$ is
a compact self-adjoint operator.  In fact, as discussed shortly, the continuity
of the covariance function ensures a stronger property: the covariance operator
is also trace class.  These issues are made precise next.

\boldheading{Covariance operators}
Let $Z:\D \times \Omega \to \R$ be a centered, second-order, mean-square continuous random field with covariance 
function $c:\D \times \D \to \R$. As before, $\D$ is assumed to be a closed and bounded 
interval in $\R$. 
The covariance operator of $Z$ is the integral operator $\C:L^2(\D) \to L^2(\D)$ defined as
\begin{equation} \label{equ:autocorr}
   [\C u](x) \defeq \int_\D c(x,y) u(y) \, \dy, \quad u \in L^2(\D).
\end{equation} 

\begin{lemma}\label{lem:Covariance}
Let $\C:L^2(\D) \to L^2(\D)$ be as in~\eqref{equ:autocorr}. Then, 
$\C$ is a compact, positive, and selfadjoint linear operator on 
$L^2(\D)$.
\end{lemma}
\begin{proof}
See appendix~\ref{appdx:proof_covariance}.
\end{proof}

Lemma~\ref{lem:Covariance} enables invoking the Spectral Theorem for compact self-adjoint
operators to conclude that $\C$ has an orthonormal basis of eigenvectors $\{ v_i
\}_{i=1}^\infty$ in $L^2(\D)$ with corresponding real non-negative eigenvalues $\{
\lambda_i\}_{i=1}^\infty$. From this point on, we assume the covariance 
function $c$ is such that $\C$ is a strictly positive operator; i.e., 
$\ip{\C u}{u} > 0$ for every nonzero $u \in L^2(\D)$, which ensures 
positive eigenvalues. 

Consider the 
spectral representation of $\C$,
\begin{equation}
\C u = \sum_{i=1}^\infty \lambda_i \ip{u}{v_i} v_i, \quad u \in L^2(\D).
\end{equation}
Invoking Mercer's Theorem, we have
\[
\int_\D c(x, x) \, \dx = 
\int_\D \sum_{i=1}^\infty \lambda_i v_i^2(x) \, \dx
= \sum_{i=1}^\infty \lambda_i \int_\D v_i(x)^2 \,\dx = \sum_{i=1}^\infty \lambda_i.
\]
Here, we have used the Lebesgue Monotone Convergence Theorem to 
interchange the integral and the infinite summation and the fact that 
the eigenvectors are normalized. The expression to the right hand side is 
the trace of $\C$. The above discussion shows that 
$\C$ has a finite trace and reveals the useful identity 
\begin{equation}\label{equ:trace}
   \trace(\C) = \int_\D c(x, x) \, \dx.
\end{equation}

\section{Spectral representation of random fields}\label{sec:analysis}
In this section, $Z:\D \times \Omega \to \R$ is a centered, second-order,
mean-square continuous random field with a strictly positive covariance
operator $\C$.  We now have all the tools to discuss a spectral representation for
$Z$ that resembles the finite-dimensional representation~\eqref{equ:KLE_FD}.
\begin{figure}[ht]\centering
   \includegraphics[width=.65\textwidth]{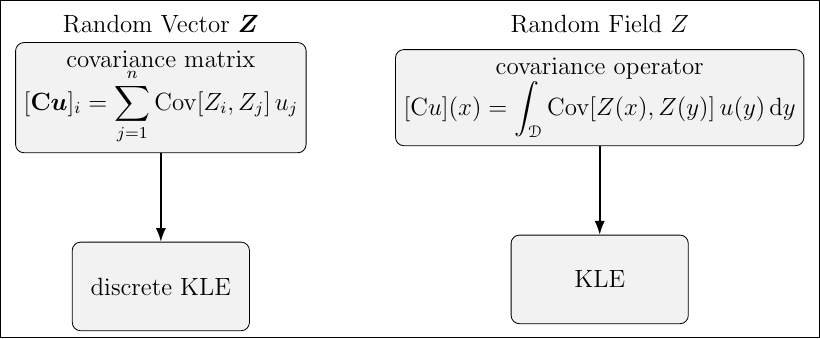}
   \caption{The discrete KLE uses the eigenpairs of $\mat{C}$ and 
   the (continuous) KLE uses those of $\C$.}
   \label{fig:diag_cov}
\end{figure}

While the discrete KLE of a random vector is tied to 
the eigenpairs of its covariance matrix, the KLE of a random field 
uses the eigenpairs of its covariance operator; see Figure~\ref{fig:diag_cov}.
The idea is to use the eigenbasis $\{v_i\}_{i=1}^\infty$ of $\C$ 
to expand
$Z$ as follows, \begin{equation} \label{equ:basic-conv} Z(x, \omega) =
\sum_{i=1}^\infty z_i(\omega) v_i(x), \quad z_i(\omega) = \int_{\mathcal{D}}
Z(x, \omega) v_i(x) \, \dx.  \end{equation} 

The key question to address here is determining  the sense in which this series
converges.  Namely, we need to understand the mode of convergence in $x$ and
$\omega$.  This  studied in Section~\ref{sec:convergence}.  In our
discussion, we first consider a simple analysis showing convergence of this
series in $L^2(\D \times \Omega)$.  This relies primarily on the properties of
$Z$ and the spectral properties of its covariance operator.  Subsequently, we
will see how invoking Mercer's theorem enables deriving a stronger result that
says
\[
   \lim_{n \to \infty} \E\Big[\Big(\Zx - \sum_{i=1}^n z_i v_i(x)\Big)^2 \Big] = 0, 
\]
uniformly in $\D$.
This also enables the following pointwise convergence result, 
which is useful from a practical perspective: for each $x \in \D$,
$\Zx = \sum_{i=1}^\infty z_i v_i(x)$, 
where the series converges in the $L^2(\Omega)$ norm.

We continue our study in Section~\ref{sec:trunc}, where we consider optimality of
the truncated KLE.  The analysis in that section is a Hilbert space analogue of
the study in Section~\ref{sec:optimality}.  
Finally, in Section~\ref{sec:GPs}, where we briefly discuss the important class of Gaussian
random fields and their KLE.  

\subsection{Convergence analysis}\label{sec:convergence}
We begin our analysis by considering 
the properties of the coefficients $\{z_i\}_{i \in \N}$ in~\eqref{equ:basic-conv}:
\begin{lemma} \label{lem:coeffs}
The coefficients $\{z_i\}_{i\in\N}$ in~\eqref{equ:basic-conv} satisfy the following:
for all $i, j \in \N$,
\begin{equation}\label{equ:KLE_coeffs_props}
\E[z_i] = 0 \quad\text{and}\quad
\E[z_i z_j] = \delta_{ij} \lambda_j.
\end{equation}
\end{lemma}
\begin{proof}
See Appendix~\ref{appdx:coeffs}.
\end{proof}

We next study the convergence of~\eqref{equ:basic-conv}.
Consider the sequence of partial sums, 
\[
   Z^{[n]}(x, \omega) = \sum_{i=1}^n z_i(\omega) v_i(x), \quad n \in \N. 
\]
Note that for almost every $\omega \in \Omega$, the 
realization $Z(\cdot, \omega)$ belongs to $L^2(\D)$. 
Since $\{v_i\}_{i\in\N}$ is an orthonormal basis of $L^2(\D)$, we have
\[
\int_{\D} |Z(x, \omega) - Z^{[n]}(x, \omega)|^2 \, \dx = \sum_{i=n+1}^\infty z_i^2(\omega).
\]
Thus, 
\begin{equation}\label{equ:tail}
\int_\Omega \int_{\D} |Z(x, \omega) - Z^{[n]}(x, \omega)|^2 \, \dx \, \P(d\omega)
\!= \! 
\mathbb{E}\left[ \sum_{i=n+1}^\infty z_i^2 \right] 
\!=\!
\sum_{i=n+1}^\infty \mathbb{E}[z_i^2]
\!=\! \sum_{i=n+1}^\infty \lambda_i.
\end{equation}
Here, the interchange of the infinite sum and the integral is justified by 
the Lebesgue Monotone Convergence Theorem.
As noted previously, $\sum_{i=1}\lambda_i < \infty$. 
Thus, by~\eqref{equ:tail}, 
\[
\lim_{n\to\infty}
\int_\Omega \int_{\D} |Z(x, \omega) - Z^{[n]}(x, \omega)|^2 \, \dx \, \P(d\omega) = 0.
\]
We have just proven the following result. 
\begin{theorem}
We have $Z = \sum_{i=1}^\infty z_i v_i$,
where the series converges in $L^2(\D \times \Omega)$.
\end{theorem}

Using Mercer's Theorem, we can prove the following stronger convergence
result.
\begin{theorem} \label{thm:KL}
The expansion
\[
\Zx = \sum_{i=1}^\infty z_i v_i(x)
\]
converges in $L^2(\Omega)$ uniformly with respect to $x \in \D$. That is,
\[
\lim_{n \to \infty} \sup_{x \in \D} \E\Big[ \big( \Zx - \sum_{i=1}^n z_i v_i(x) \big)^2 \Big] = 0.
\]
The random coefficients $z_i$ are given by $z_i = \int_\D \Zx v_i(x) \, \dx$ and satisfy~\eqref{equ:KLE_coeffs_props}.
\end{theorem}

\begin{proof}
Since the covariance operator $\C$ of $Z$ is a strictly positive compact self-adjoint operator, 
the Spectral Theorem guarantees the existence of an orthonormal basis 
$\{v_i\}_{i=1}^\infty$ of $L^2(\mathcal{D})$ consisting of eigenvectors a 
of $\C$. Let $\{\lambda_i\}_{i=1}^\infty$ be the corresponding eigenvalues, 
ordered according to $\lambda_1 \ge \lambda_2 \ge \dots > 0$.
Consider
\[
\eps_n(x) \defeq \E\Big[ \big( \Zx - \sum_{i=1}^n z_i v_i(x) \big)^2 \Big], \quad x \in \D. 
\]
The rest of the proof is devoted to showing 
$\lim\limits_{n \to \infty}\eps_n(x) = 0$ uniformly in $\D$. 

We first note that
\begin{equation} \label{equ:exp-I}
\begin{aligned}
\eps_n(x) &= \E\Big[\Big(\Zx - \sum_{i=1}^n z_i v_i(x)\Big)^2 \Big] \\ 
          &= \E[\Zx^2] - 2 \E[\Zx \sum_{i=1}^n z_i v_i(x)] + 
                              \E[\sum_{i,j=1}^n z_i z_j v_i(x)v_j(x)]
\end{aligned}
\end{equation}
Now, $\E[\Zx^2] = c(x,y)$ with $c$ as in~\eqref{equ:autocorr}. Moreover,
\begin{align}
\E[\Zx \sum_{i=1}^n z_i v_i(x)] &= 
   \E\Big[\Zx \sum_{i=1}^n \Big(\int_\D \Zy v_i(y) \, \dy\Big) v_i(x)\Big]
   \notag \\ 
   &= \sum_{i=1}^n \Big(\int_\D \E[\Zx \Zy] v_i(y) \, \dy\Big) v_i(x) 
   \notag \\
   &= {\sum_{i=1}^n [\C v_i](x) v_i(x)} \\
   &= \sum_{i=1}^n \lambda_i v_i(x)^2. \label{equ:exp-II} 
\end{align}
Through a similar argument, we can show that
\begin{equation} \label{equ:exp-III}
\E\big[\sum_{i,j=1}^n z_i z_j v_i(x)v_j(x)\big] = \sum_{i=1}^n \lambda_i v_i(x)^2.
\end{equation}
Therefore, by~\eqref{equ:exp-I},~\eqref{equ:exp-II}, and~\eqref{equ:exp-III}, 
we have
$\eps_n(x) = c(x,x) - \sum_{i=1}^n \lambda_i v_i(x)v_i(x)$. Thus, 
by Mercer's theorem, $\eps_n(x) \to 0$ uniformly in $\D$.
\end{proof}

The coefficients $\{z_i\}_{i=1}^\infty$ are mean zero and 
pairwise uncorrelated. 
As done in the finite-dimensional case in Section~\ref{sec:fd-study}, we 
consider the standardized coefficients 
\[
   \xi_i = \frac{1}{\sqrt{\lambda_i}} z_i, \quad i \in \N.
\]

Since the KLE converges uniformly in $\D$, we have
pointwise convergence. 
We thus state the following familiar form of the 
KLE: for every $x \in \D$, the series
\begin{equation}\label{equ:convL2}
Z(x, \omega) = \sum_{i=1}^\infty \sqrt{\lambda_i} \xi_i(\omega) v_i(x)
\end{equation}
converges in $L^2(\Omega)$.

\subsection{Truncation and optimality}\label{sec:trunc}
Consider the truncated KLE,
\begin{equation}\label{equ:truncatedKLE}
Z^{[r]}(x, \omega) \defeq \sum_{i=1}^r \sqrt{\lambda_i} \xi_i(\omega) v_i(x),
\end{equation} 
for $r \in \N$.
Determining a truncation level $r$ may be done analogously to the approach considered
in Section~\ref{sec:fd-study}.  This will be discussed
further later. Here, we discuss an optimality result for the KLE, which is 
the Hilbert space analogue of Proposition~\ref{prp:optimality}.

Given a truncation level $r \in \N$, we represent  
the truncated KLE as  
\begin{equation}\label{equ:KLEr}
\Zxr{r}(\omega) = \sum_{i=1}^r \ip{Z(\cdot, \omega)}{v_i} v_i(x).
\end{equation}
Further, 
given an orthonormal basis,
$U = \{u_i\}_{i=1}^\infty \subset L^2(\D)$, we consider the series 
expansion 
$\Zx(\omega) = \sum_{i=1}^\infty \ip{Z(\cdot, \omega)}{u_i} u_i(x)$, 
and the corresponding truncated version,
\[
\Zxr{r,U}(\omega) = \sum_{i=1}^r \ip{Z(\cdot, \omega)}{u_i} u_i(x).
\]
The following result states the well-known~\cite{Ghanem,Knio10} optimality of the
truncated KLE.
\begin{theorem}\label{thm:KLE_optimality}
For every orthonormal basis $U = \{u_i\}_{i=1}^\infty$ 
of $L^2(\D)$ and $r \in \N$,
\[
\E\Big[ \int_\D \big(\Zx -  \Zxr{r}\big)^2 \, \dx\Big] \leq 
\E\Big[ \int_\D \big(\Zx -  \Zxr{r, U}\big)^2 \, \dx\Big]. 
\]
\end{theorem}
\begin{proof}
By a similar calculation to that in~\eqref{equ:tail},
\[
\E\Big[ \int_\D \big(\Zx -  \Zxr{r}\big)^2 \, \dx\Big] = \sum_{i=r+1}^\infty \lambda_i
= \trace(\C) - \sum_{i=1}^r \lambda_i.
\]
We next note that 
\begin{multline*}
\E\Big[ \int_\D \big(\Zx -  \Zxr{r, U}\big)^2 \, \dx\Big]
= 
\E\Big[ \int_\D \big( \sum_{i=r+1}^\infty \ip{Z(\cdot, \omega)}{u_i} u_i(x)  \big)^2 \, \dx\Big]\\
=
\E\Big[ \sum_{i=r+1}^\infty \ip{Z(\cdot, \omega)}{u_i}^2 \Big]
=
\sum_{i=r+1}^\infty \E\big[ \ip{Z(\cdot, \omega)}{u_i}^2 \big]
=
\sum_{i=r+1}^\infty \ip{\C u_i}{u_i}.
\end{multline*}
In the last step  
we used $\E\big[ \ip{Z(\cdot, \omega)}{u_i}^2 \big] 
= \ip{\C u_i}{u_i}$, which 
can be justified by similar arguments as those in 
the proof of Lemma~\ref{lem:Covariance}; cf.\ Appendix~\ref{appdx:proof_covariance}.
Therefore, 
\[
\begin{aligned}
\E\Big[ \int_\D \big(\Zx -  \Zxr{r, U}\big)^2 \, \dx\Big] - 
&\E\Big[ \int_\D \big(\Zx -  \Zxr{r}\big)^2 \, \dx\Big]\\ 
&= \sum_{i=r+1}^\infty \ip{\C u_i}{u_i} - \sum_{i=r+1}^\infty \lambda_i \\
&= \trace(\C) - \sum_{i=1}^r \ip{\C u_i}{u_i} - \trace(\C) + \sum_{i=1}^r \lambda_i \\
&= \sum_{i=1}^r \lambda_i - \sum_{i=1}^r \ip{\C u_i}{u_i}  \geq 0, 
\end{aligned}
\]
where the final inequality follows from Proposition~\ref{prp:KyFanHilbert}.
\end{proof}

\subsection{Gaussian random fields}\label{sec:GPs}
Here, we briefly consider the Gaussian
random fields, which play a central role in probability theory and uncertainty
quantification.  As discussed shortly, Gaussian random fields possess additional
structural properties that make them particularly convenient to work with.

A random field $Z:\D \times \Omega \to \R$ is called a Gaussian 
random field, or a Gaussian process, if for every $\{x_1, \ldots, x_N\} \subset \D$, 
the vector 
\begin{equation}\label{equ:Zdef}
\vec{Z}(\omega) \defeq [Z(x_1, \omega) \; Z(x_2, \omega) \; \cdots \; Z(x_N, \omega)]^\tran
\end{equation}
follows a multivariate Gaussian distribution. The KLE of a 
mean-square continuous Gaussian random field has two key features in addition 
to the ones proven in the previous subsections. Namely, for such  
Gaussian random fields 
\begin{enumerate}[label=(\roman*)]
\item the coefficients $\{z_i\}_{i=1}^\infty$ are independent standard normal random variables;
\item for each $x \in \D$, the series~\eqref{equ:convL2} converges almost surely.
\end{enumerate}
We refer to~\cite[Chapter 1]{AshGardner75} for proofs of these results.  Note
that (i) makes generating realizations of a Gaussian random field using a 
truncated KLE quite straightforward: upon computing the eigenpairs of the 
covariance operator, we only need to generate $r$ standard normal random
variables and evaluate the truncated KLE.  

Another key property of a Gaussian random field is its unique 
characterization by its mean and covariance
function; see, e.g.,~\cite[Section 1.6]{Adler10}. In particular, given a continuous covariance function 
$c:\D \times \D \to \R$, 
we can define a centered Gaussian random field $Z$. 
Note that, in this case,
the random vector $\vec Z$ defined 
in~\eqref{equ:Zdef} is distributed according to the multivariate Gaussian 
law $N(\vec{0}, \mat{C})$ with 
$C_{ij} = c(x_i, x_j)$ for $i, j \in \{1, \ldots, N\}$. For example, given  
$x_1$, $x_2$ in $\D$,
\[
\begin{bmatrix} Z(x_1, \omega) \\ Z(x_2, \omega)\end{bmatrix}
\sim 
N\left(
\begin{bmatrix} 0 \\ 0 \end{bmatrix},
\begin{bmatrix}
c(x_1, x_1) & c(x_1, x_2) \\
c(x_2, x_1) & c(x_2, x_2)
\end{bmatrix}
\right).
\]

\section{Discretization}\label{sec:discretization}
To form the KLE of a random field $Z$, with covariance function 
$c$, 
we need to solve the following eigenvalue problem:
find $\{\lambda_i\}_{i \in \N}$ and the 
corresponding orthonormal basis of eigenvectors 
$\{v_i \}_{i \in \N} \subset L^2(\D)$ that satisfy 
\begin{equation}\label{equ:eigenvalue_problem}
\int_\D c(x, y) v_i(y) \,\dy = \lambda_i v_i(x),  \quad i \in \N.
\end{equation}

To enable numerical computations, we need to discretize the eigenvalue
problem~\eqref{equ:eigenvalue_problem}. There are several approaches for doing
this.  We discuss one approach here---\Nystrom{} method.  In this approach, we
use quadrature to
discretize the integral in~\eqref{equ:eigenvalue_problem}.  
Consider a quadrature formula 
with nodes and weights $\{(x_k, w_k)\}_{k=1}^n$ used 
to approximate $\int_\D f(x) \,\dx$, where 
$f:\D \to \R$ is an integrable function: 
\[
\int_\D f(x) \, \dx \approx \sum_{k=1}^n w_k f(x_k).
\]

Discretizing the integral in~\eqref{equ:eigenvalue_problem} and 
evaluating the resulting equation at the quadrature nodes, we obtain the discretized eigenvalue problem
\begin{equation}\label{equ:discrete_eigenproblem_components}
   \sum_{l = 1}^n w_l c(x_k, x_l) v_i(x_l) = \lambda_i v_i(x_k), \quad k = 1, \ldots, n.
\end{equation}
Note that, with a slight abuse of notation, we continue to use $\{(\lambda_i, v_i)\}_{i=1}^n$
to denote the approximate eigenpairs. 

To make matters concrete, we next describe the discretized eigenvalue problem 
in matrix-vector notation. 
Let the vector $\vec{v}_i$ be defined by 
\[
\vec{v}_i = [v_i(x_1) \; v_i(x_2) \; \cdots  \; v_i(x_n)], 
\]
and 
define the matrices $\mat{C}$ and $\mat{W}$ according to 
$C_{kl} = c(x_k, x_l)$, $k, l \in \{1, \ldots, n\}$, and 
$\mat{W} = \operatorname{diag}(w_1, w_2, \ldots, w_n)$, respectively.
We can state the discretized eigenvalue problem as follows: 
\[
    \mat{C}\mat{W} \vec{v}_i = \lambda_i \vec{v}_i, \quad
    i = 1, 2, \ldots, n.
\]
This can be rewritten as the symmetric eigenvalue problem,
\begin{equation}\label{equ:discrete_eigenproblem_symmetric}
    \mat{W}^{1/2} \mat{C} \mat{W}^{1/2} \tilde{\vec{v}}_i = \lambda_i \tilde{\vec{v}}_i,
\end{equation}
with $\tilde{\vec{v}}_i = \mat{W}^{1/2}\vec{v}_i$. 
Upon solving this eigenvalue problem,
we obtain eigenvalues $\{\lambda_i\}_{i=1}^n$ and eigenvectors $\{\tilde{\vec{v}}_i\}_{i=1}^n$, with
$\tilde{\vec{v}}_i^\tran \tilde{\vec{v}}_j = \delta_{ij}$.  We can then obtain the eigenvectors
$\{\vec{v}_i\}_{i=1}^n$ using
$\vec{v}_i = \mat{W}^{-1/2} \tilde{\vec{v}}_i$. Also note that,
$\vec{v}_i^\tran \mat{W} \vec{v}_j = \delta_{ij}$ for $i,j \in \{1, \ldots, n\}$.

The eigenpairs obtained by solving the discretized eigenvalue problem approximate
those of the continuous eigenvalue problem~\eqref{equ:eigenvalue_problem}.
The present discussion raises a number of practical considerations.  These
include the choice of the quadrature formula and its accuracy.  Such choices
must be made in accordance with the number of desired terms in the truncated
KLE. We revisit these issues in Section~\ref{sec:computational_considerations}.

\section{Computational modeling with Gaussian random fields}\label{sec:numerics}
Gaussian random fields are commonly used in uncertainty quantification.  When
using such fields in practice, careful attention must be paid to an
interconnected set of modeling, mathematical, and computational 
considerations.
We highlight some of these issues in
Sections~\ref{sec:modeling_considerations} through \ref{sec:computational_considerations}.
To provide further insight, we consider several illustrative
numerical experiments in Section~\ref{sec:gaussian_numerics}.  It is important to note
that while this section is focused on Gaussian random fields, the discussions
that follow have fundamental implications for computational modeling with random
fields in general. 

\subsection{Modeling considerations}\label{sec:modeling_considerations}
Such considerations are inherently application-dependent.
To illustrate some key modeling issues, we focus on a specific 
example: modeling the permeability field of a heterogeneous medium.
As before, we 
focus on a one-dimensional domain $\D = [a, b]$. Following common practice,
we define the permeability field as the function 
$\kappa: \D \times \Omega \to \R$, given by 
$\kappa(x, \omega) = \exp(Y(x, \omega))$, where 
\begin{equation}\label{equ:random_field}
    Y(x, \omega) = \bar{z}(x) + Z(x, \omega). 
\end{equation}

Here, $\bar{z}$ is the mean of the process and $Z(x, \omega)$ is a centered mean
square continuous Gaussian random field. To specify  $\bar{z}$, which describes 
the mean field behavior, one may rely on prior knowledge or modeling
assumptions.  On the other hand, $Z$ encodes our modeling assumptions
regarding the fluctuations of the material properties.  
Note also that the present construction ensures positivity of the permeability field $\kappa$.

Since $Z$
is mean zero, it is characterized uniquely by its covariance function.  To
ensure mean-square continuity, we choose a continuous covariance function. 
There are various options for doing so. A simple choice is the
exponential covariance function,
\begin{equation}\label{equ:classical}
c(x, y) \defeq \sigma^2 \exp\left( -\frac{|x - y|}{\ell}\right), \quad x, y \in \D.
\end{equation}

Here, $\sigma^2 > 0$ specifies the pointwise variance of the field, and $\ell >
0$ is a scale parameter controlling the correlation length. From a physical
perspective, $\sigma$ and $\ell$ can be used to model heterogeneous properties
of a medium.  By controlling the size of the pointwise variance, $\sigma$
specifies level of contrast or variations in the materials properties.  A
large $\sigma$ indicates the high likelihood of large variations in the
permeability field.  On the other hand, the correlation length parameter
controls how rapidly the material properties can change locally.  Note
that the correlation length should be considered relative to the size of the
domain.  A small $\ell$ relative to the domain size (characteristic length
scale) models a highly disordered medium where material properties can change
drastically in nearby points. 

\subsection{Mathematical considerations}
Consider the random field $Y(x, \omega)$ 
in~\eqref{equ:random_field} more broadly and beyond its specific application in
modeling heterogeneous material properties.  The covariance function of $Z$ defined
in~\eqref{equ:classical} yields a Gaussian random field with almost surely
continuous (but nowhere differentiable) realizations.  There are various other choices for covariance
functions. A widely used option is to choose from among the family of
Mat\'ern covariance functions.  We refer to~\cite[Chapter
4]{RasmussenWilliams06} or~\cite[Chapter 2]{Stein99} for details. The class of
Mat\'ern covariance functions enable controlling regularity properties of 
the realizations---one can specify varying degrees of differentiability for the
realization of $Z$.  Note that the covariance
function~\eqref{equ:classical} is indeed a member of the Mat\'ern family of
covariance function.  For further discussions and examples of covariance
functions, we refer to \cite[Chapter 4]{RasmussenWilliams06}.

The mean function $\bar z$ also deserves some discussion.
To ensure mean-square continuity of $Y$,
in addition to mean 
square continuity of $Z$, 
we need 
$\bar{z}$ to be a continuous function on $\D$. Note specifically that 
\[
\E\Big[\big(Y(x+\eps,\cdot) - Y(x, \cdot)\big)^2\Big] 
= \big(\bar{z}(x+\eps) - \bar{z}(x)\big)^2 + \E\Big[\big(Z(x+\eps, \cdot) - Z(x, \cdot)\big)^2\Big].
\] 
Thus, assuming $\bar{z} \in C(\D)$ and that $Z$ is a centered mean-square
continuous random field, we have that $\E\big[\big(Y(x+\eps,\cdot) - Y(x, \cdot)\big)^2\big] \to
0$ as $\eps \to 0$. 
More broadly, to ensure $Y(x, \omega)$ has a desired degree of spatial 
regularity, $\bar{z}$ and $Z$ must both satisfy that regularity requirement 
independently. For example, when using a Mat\'ern covariance function that
ensures the realizations of $Z$ are almost surely differentiable, we must choose
a differentiable $\bar{z}$ as well. 

The present discussion points to a fundamental difference between working with
Gaussian random vectors and Gaussian random fields. In the case of random
vectors, the covariance matrix is a purely statistical construct---it describes
the pointwise variance and correlation structure. On the other hand, working
with random fields adds a new dimension: the covariance function determines the
spatial regularity of the realizations in addition to the statistical properties
of the field. When incorporating random fields in PDE models, it is important to
account for such  regularity properties.

\subsection{Computational considerations}\label{sec:computational_considerations}
In practical computations, one considers truncated KLEs of 
random fields. We thus need a pragmatic approach to specify a truncation level. 
Furthermore, to enable computations, the KLE must be discretized.  The choices
of truncation and discretization scheme are closely related. Namely, the desired
truncation level dictates the level of resolution in one's choice of
discretization. Conversely, a fixed discretization limits the number of terms in
the KLE that can be approximated reliably.  We
discuss these issues briefly in this section.

The computational cost of solving the (discretized) eigenvalue problem for
the eigenpairs of the covariance operator is also a key consideration for
problems on complex two- or three-dimensional domains.  In such cases, the
discretized covariance operator is a dense high-dimensional matrix, and a
large-scale eigenvalue problem must be solved.  While a detailed discussion of
the associated numerical methods is beyond the scope of our discussion, 
addressing this challenge is crucial in large-scale applications.
We mention a couple of references that provide complementary
perspectives in this direction.  We point to~\cite{Khoromskij09}, which utilizes
$\mathcal{H}$-matrices to enable fast computation of the KLE. See
also~\cite{LindgrenRue11}, which establishes a tractable approach for computing
with Gaussian random fields with Mat\'ern covariance functions by exploiting
their connections to Gaussian Markov random fields and stochastic PDEs.

\boldheading{Truncation}
Here, we focus on determining a truncation level $r$ for a centered
random field $Z$ as in~\eqref{equ:truncatedKLE}.  
The average pointwise variance of $Z$ is given by 
\[
    \frac{1}{b-a} \int_\D \Var[Z(x, \cdot)] \,\dx 
    =  \frac{1}{b-a} \int_\D c(x, x) \, \dx 
    = \frac{\trace(\C)}{b-a}, 
\]
where $\C$ is the covariance operator of $Z$, and 
we have used~\eqref{equ:trace} in the last step.
It is also straightforward to show that $\Var[Z^{[r]}(x, \cdot)] = \sum_{i=1}^r \lambda_i v_i(x)^2$.
Thus, the average variance of $Z^{[r]}$ is given by 
\[
\frac{1}{b-a} \int_\D \Var[Z^{[r]}(x, \cdot)] \,\dx =  
\frac{1}{b-a}\sum_{i=1}^r \lambda_i \int_\D v_i(x)^2\ \,\dx = 
\frac{1}{b-a} \sum_{i=1}^r \lambda_i, 
\]
where 
we have also used the fact that $\int_\D v_i^2(x) \, \dx = \ip{v_i}{v_i} = 1$.

The present discussion shows that the fraction of the average variance 
captured by the $r$-term truncation of the KLE is given by 
\begin{equation}\label{equ:ratio}
\rho(r) \defeq 
\frac{\sum_{i=1}^r \lambda_i}{\sum_{i=1}^\infty \lambda_i}
=
\frac{\sum_{i=1}^r \lambda_i}{\trace(\C)}
=
\frac{\sum_{i=1}^r \lambda_i}{\int_\D c(x, x)\, \dx}.
\end{equation}

The final expression in~\eqref{equ:ratio} is convenient for computations.
Namely, given an analytic expression for the covariance function, computing an
accurate estimate of $\int_\D c(x, x) \, \dx$ is easier than estimating the
denominator by computing a large number of eigenvalues.  In practice, we may
choose $r$ such that $\rho(r)$ is larger than a desired ratio. For example, 
$\rho(r) > 0.95$ implies over 95\% of the average
variance is captured by $Z^{[r]}$.  Note that this truncation
strategy is not restricted to Gaussian random fields and applies to
any second-order, mean-square continuous random fields.

\boldheading{Discretization}
As discussed in Section~\ref{sec:discretization}, to compute the spectral
decomposition of the covariance operator numerically, we need to discretize the eigenvalue
problem.  This can be done using several approaches. One option is \Nystrom{}
method considered in Section~\ref{sec:discretization}.  Another possibility is 
to perform a  
finite element discretization; this is natural when using random
fields to model uncertain field quantities appearing in PDEs 
on complex geometries.  For an overview of different approaches to
discretization of KLEs, we refer to~\cite{BetzPapaioannouStraub14}. Our
discussion will be limited to \Nystrom{} method. This is sufficient to expose
the key computational considerations regarding the discretization of KLEs. 

The resolution of the discretization is linked closely to the number of terms
retained in the truncated KLE~\eqref{equ:truncatedKLE}. In this case, we
approximate the eigenpairs $\{(\lambda_i, v_i)\}_{i=1}^r$ by solving a
discretized eigenvalue problem of the form~\eqref{equ:discrete_eigenproblem_symmetric}.
Consequently, we need a sufficiently fine grid to accurately compute the leading
eigenvalues and the corresponding eigenvectors. This is illustrated numerically
in the next section. 

\subsection{Numerical illustrations}\label{sec:gaussian_numerics}
Here, we provide a few numerical experiments to illustrate the issues discussed
in the previous subsections.  We focus on a specific example. We let $\D = [0,
1]$ and consider a centered Gaussian random field $Z(x, \omega)$ with covariance
function defined as in~\eqref{equ:classical} with 
$\sigma = 1$.
The eigenvalues and eigenvectors of 
the resulting covariance operator can be characterized analytically~\cite{Ghanem}. 
However, in the present numerical study, we rely on \Nystrom{} method 
to further illustrate the numerical issues in computing the KLE in practice. 

\boldheading{The correlation structures}
With our choice of the covariance function, 
$\Var[\Zx] = c(x, x) = 1$, for each $x \in \D$. Thus, 
for $x, y \in \D$,
\[
   \Cor[\Zx \Zy] = \E[ \Zx \Zy] = c(x, y).
\] 
\begin{figure}[!ht]\centering
   \includegraphics[width=.45\textwidth]{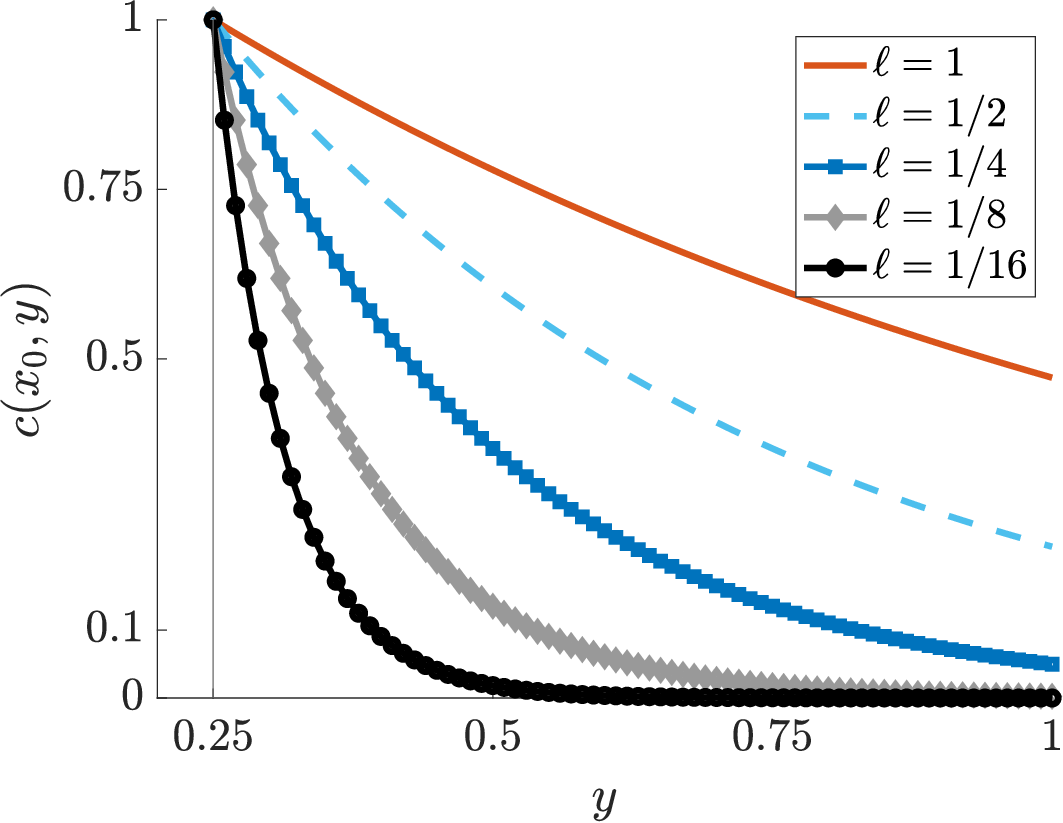}
   \includegraphics[width=.45\textwidth]{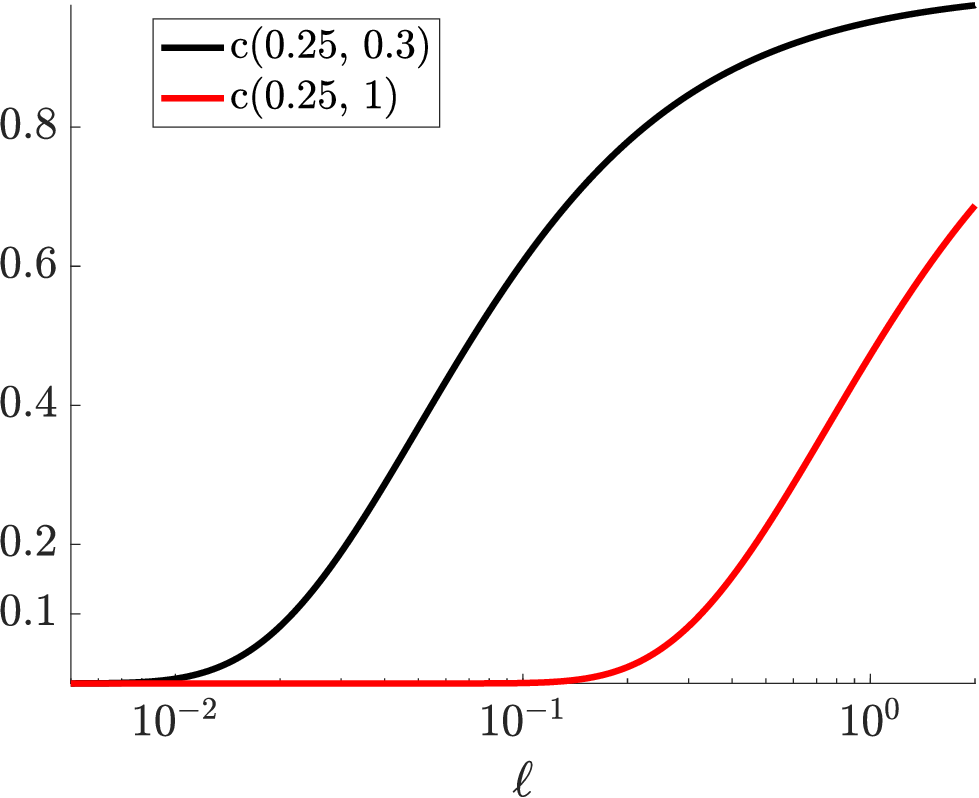}
   \caption{Left: $c(x_0, y)$, with $x_0 = 0.25$ and $y \in [x_0, 1]$, for different 
   choices of $\ell$. Right: 
   $c(0.25, 0.3)$ and $c(0.25. 1)$ 
   as a function of the correlation length parameter $\ell$.}
   \label{fig:corstudy}
\end{figure}

We begin by studying the impact of the correlation length parameter $\ell$ on $Z$.
In Figure~\ref{fig:corstudy}~(left), we fix $x_0 = 0.25$ and consider the
function $c(x_0, y)$, for $y \in [x_0, 1]$, for several choices of $\ell$.
Note that for $\ell = 1$, $\ZZ{x_0}$ is strongly correlated to $\Zy$ even for
$y$ near 1. On the other hand, in the case of $\ell = 1/16$, $\ZZ{x_0}$ 
and $\Zy$ are nearly uncorrelated for $y \geq 0.5$. A complementary perspective is
provided in Figure~\ref{fig:corstudy}~(right), where we depict $c(0.25, 0.3)$
and $c(0.25, 1)$ as the correlation length parameter $\ell$ varies.  Note
that $\ZZ{0.25}$ and $\ZZ{0.3}$ are nearly uncorrelated for small $\ell$.  On
the other hand, for large $\ell$, 
even $\ZZ{0.25}$ and $\ZZ{1}$ are strongly correlated.
These results show that a large $\ell$ results in significant
correlation between $\Zx$ and $\Zy$ even for distant $x$ and $y$, whereas a small $\ell$
yields arbitrarily small correlations for nearby points.

We next illustrate the role of the correlation length by considering the
realizations of the random field, which we generate using truncated KLEs with
100 terms.
To compute these KLEs numerically, we
rely on \Nystrom{} method, and discretize the eigenvalue
problem~\eqref{equ:eigenvalue_problem} with a composite trapezoid rule
on a sufficiently fine grid.
In Figure~\ref{fig:realizations}, we depict 5 realizations of the random field
with $\ell = 1/16$~(left) and $\ell = 1$~(right). Note that the smaller correlation length
results in drastic local variations in the values of $Z(x, \omega)$.
This also illustrates a point made in
Section~\ref{sec:modeling_considerations}: when using random fields to model
physical material properties, the correlation length parameter can be used to
control the degree of heterogeneity of the medium. We can also control the
contrast in the constituent material properties 
statistically, by using a different pointwise variance $\sigma^2$ or even a
space-dependent $\sigma^2$. 
\begin{figure}[!ht]\centering
   \includegraphics[width=.4\textwidth]{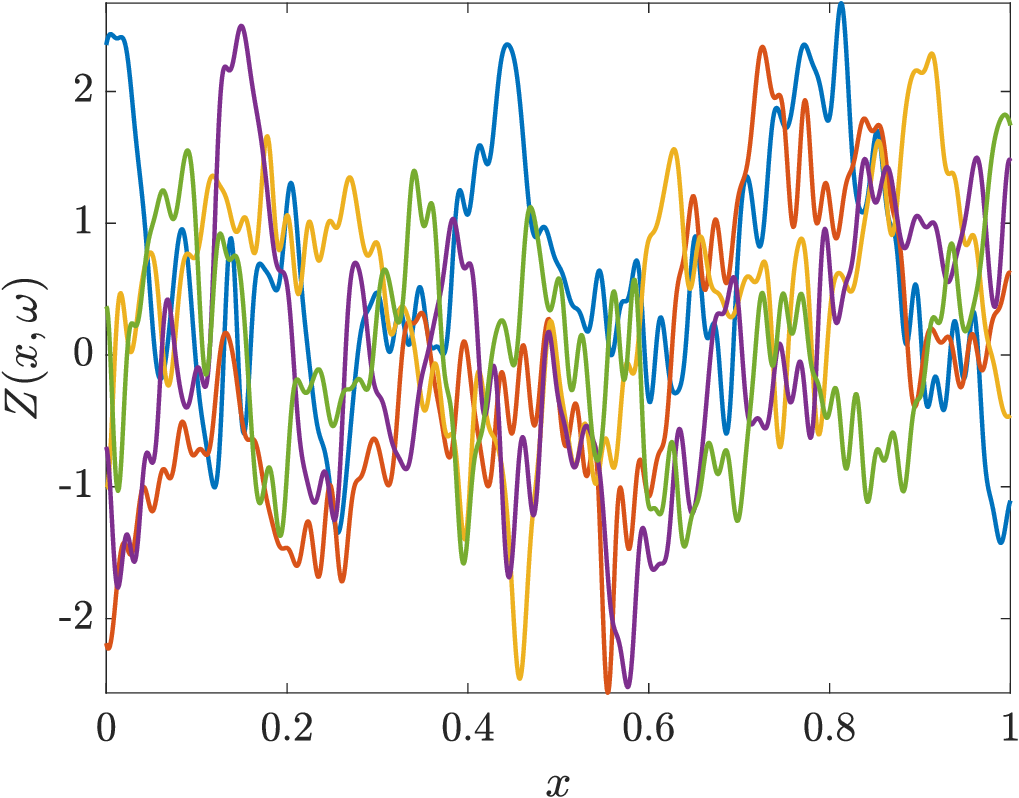}
   \includegraphics[width=.4\textwidth]{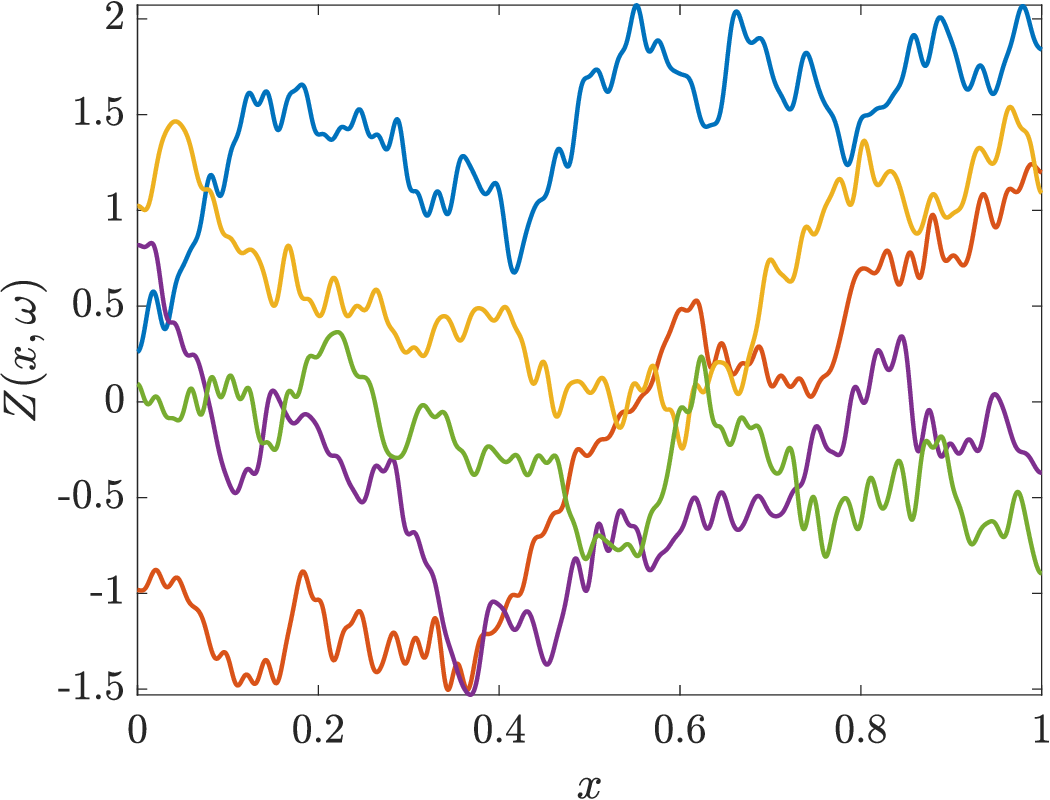}
   \caption{Realization of Gaussian random fields with $\ell = 1/16$~(left) and $\ell = 1$~(right), 
   generated using truncated KLEs with 100 terms.}
   \label{fig:realizations}
\end{figure}

\begin{figure}[ht]\centering
   \includegraphics[width=.45\textwidth]{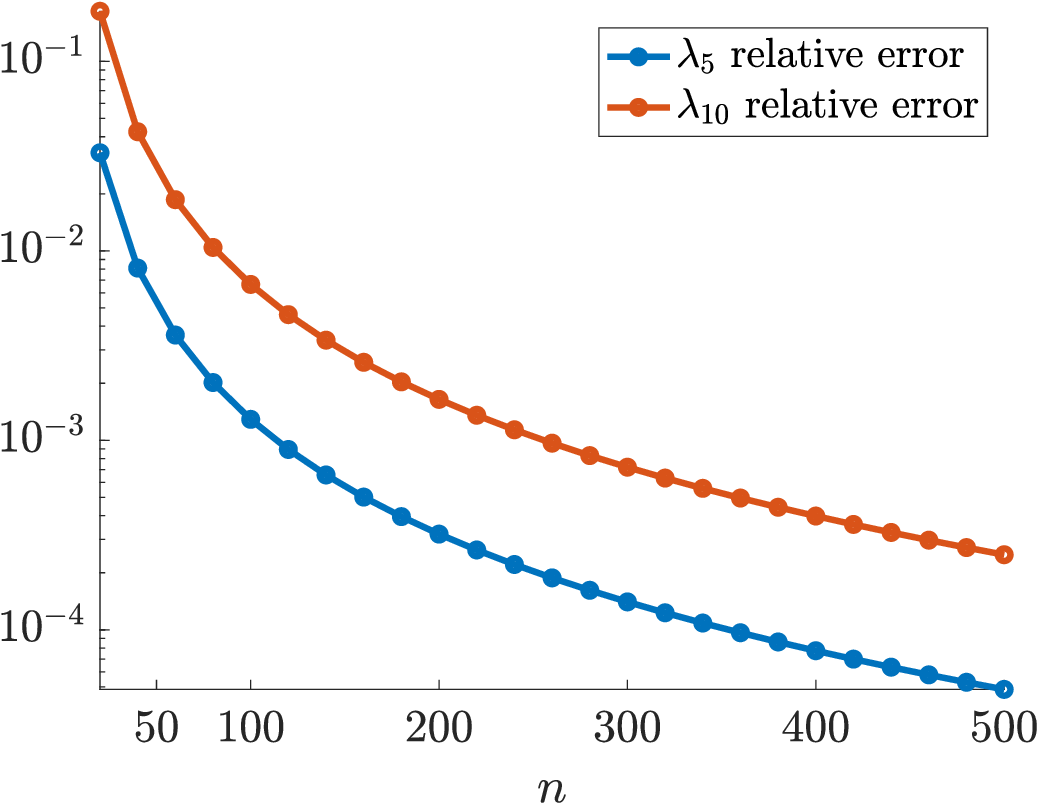}
   \caption{
   Relative errors in the eigenvalues $\lambda_{5}$ and $\lambda_{10}$ 
   computed with different numbers of quadrature nodes.}
   \label{fig:eigenvalue} 
   \end{figure} 

\boldheading{The impact of discretization on the eigenpairs of $\C$}
We next illustrate how the grid resolution affects the accuracy of the computed
eigenpairs of the covariance operator $\C$.  As before, we compute these
eigenpairs numerically using \Nystrom{} method with a composite trapezoid
rule.  

We consider the case of $\ell = 1$ and focus on $(\lambda_k, v_k)$ with
$k \in \{5, 10\}$.  
In Figure~\ref{fig:eigenvalue}, we report the relative error in computing these
eigenvalues as the number $n$ of grid points increases.  To approximate these
errors, we use highly accurate estimates of $\lambda_5$ and $\lambda_{10}$,
computed on a fine grid with 2000 nodes.  Observe that computing $\lambda_{5}$
with $n = 100$ results in a relative error of less than $10^{-3}$.  On the other
hand, achieving the same level of accuracy in the computed $\lambda_{10}$ 
requires a substantially finer grid.   
In Figure~\ref{fig:eigenvector}, we consider the corresponding eigenvectors,
computed using low- and high-resolution grids with $n = 20$ and $n = 500$ nodes, 
respectively.  Observe that resolving the high-frequency
oscillations of $v_{10}$ requires a sufficiently fine grid.  
\begin{figure}[!ht]\centering
\includegraphics[width=.4\textwidth]{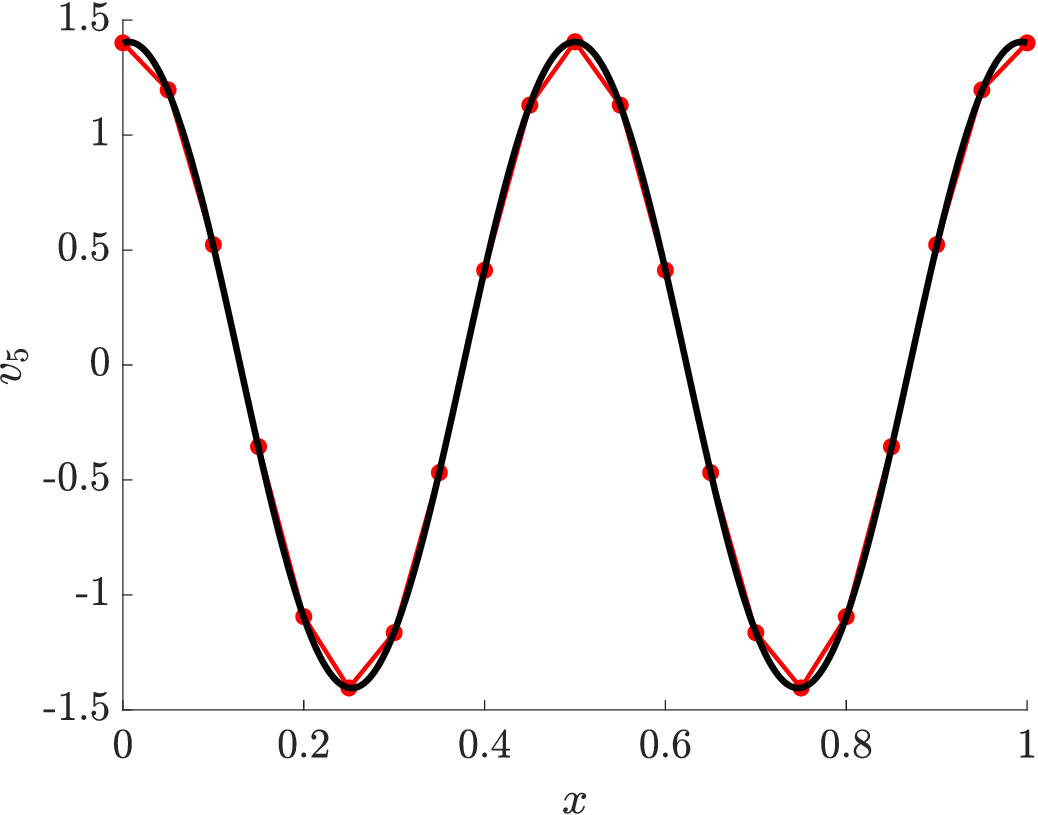}
\includegraphics[width=.4\textwidth]{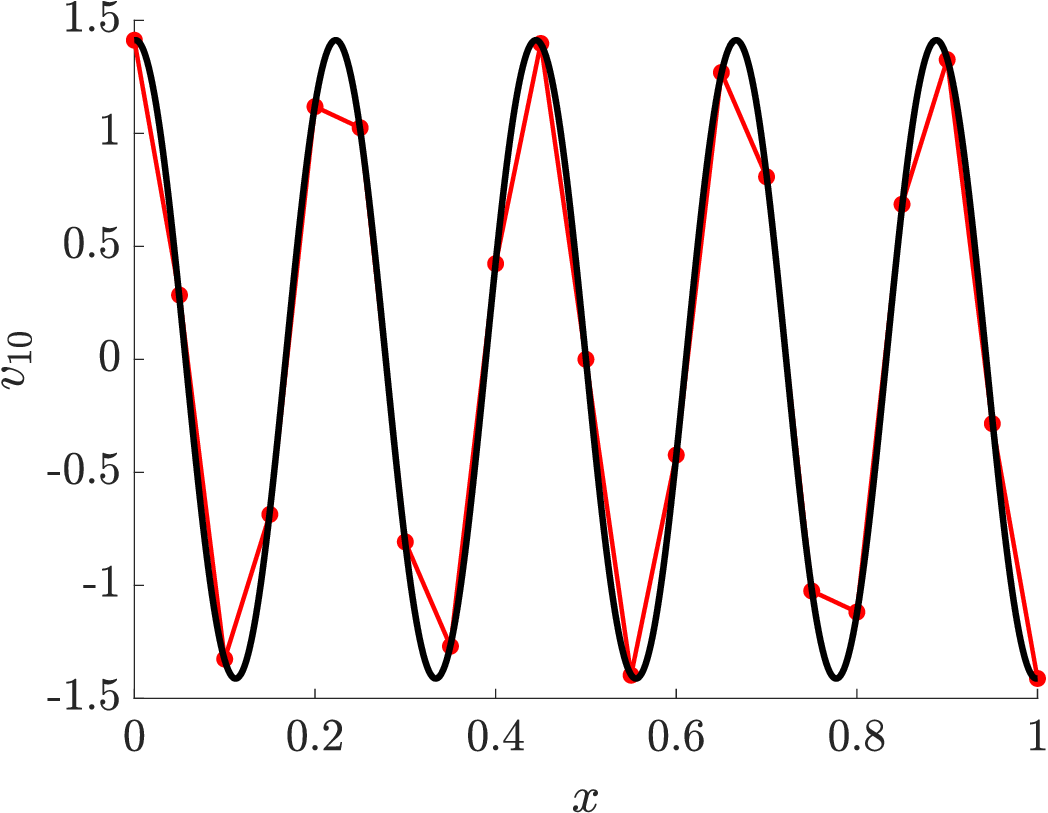}
\caption{Eigenvectors $v_{5}$ and $v_{10}$ estimated with a low- (red) and
high-resolution (black) grid.} \label{fig:eigenvector} \end{figure}

The present brief illustration indicates
that the resolution of the discretization scheme must be commensurate with the
number of terms retained in the KLE---the more terms retained, the 
larger the value of $n$ necessary for accurate estimation of the higher order 
eigenpairs. In practice, a careful grid refinement study is needed to ensure 
sufficient accuracy is achieved.

\boldheading{Choosing the truncation level}
\begin{figure}[ht]\centering
   \includegraphics[width=.4\textwidth]{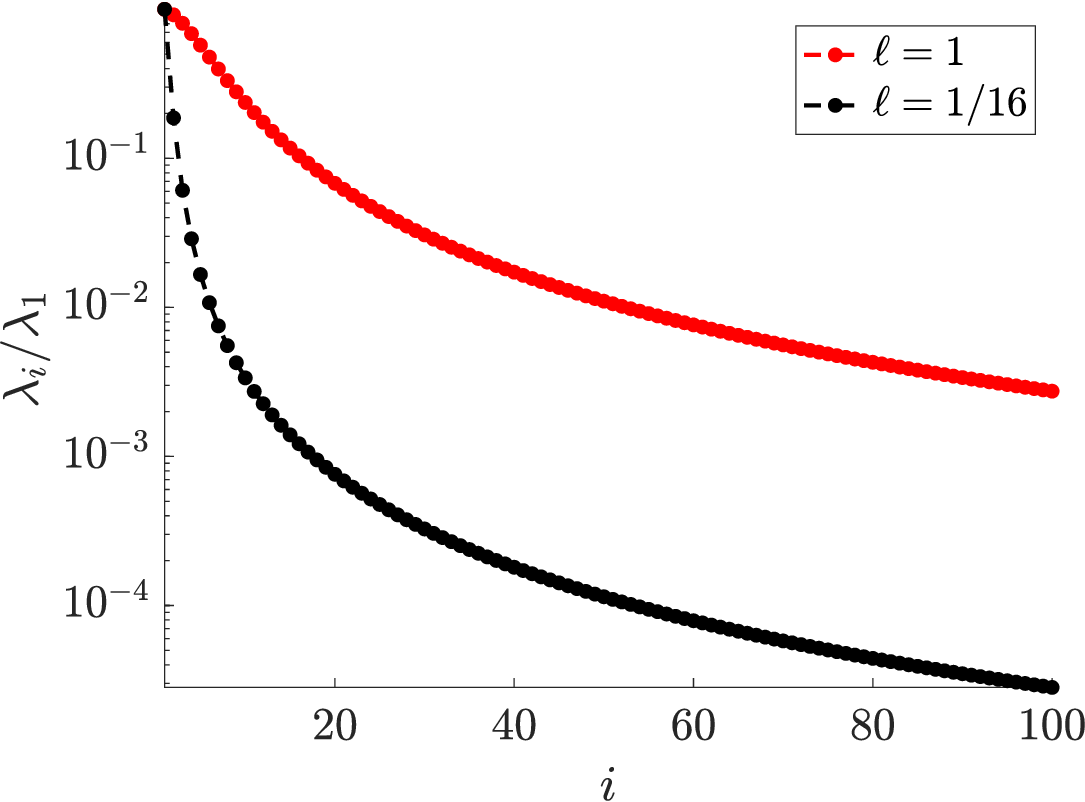}
   \caption{The first 100 normalized 
   eigenvalues of the covariance operator.}
   \label{fig:eigstudy}
\end{figure}
In Figure~\ref{fig:eigstudy}, we report the first 100 normalized eigenvalues of
the covariance operator when $\ell = 1$ and $\ell = 1/16$.  Observe that a
smaller correlation length parameter results in a slower spectral decay. To
assess how many terms we need to retain in the KLE, we can rely on the
truncation strategy outlined in Section~\ref{sec:trunc}.  
Specifically, we consider the ratio~\eqref{equ:ratio}.  For the
present choice of the covariance function, $\int_0^1 c(x, x) \, \dx = 1$, and thus
$\rho(r) = \sum_{i=1}^r \lambda_i$.  
As an example, when $\ell=1$, we need 21 terms in the KLE to capture about 
99\% of the average variance. On the other hand, in the case of $\ell=1/16$, we
need $r = 300$ to achieve the same effect. 

\begin{figure}[ht]\centering
   \includegraphics[width=.65\textwidth]{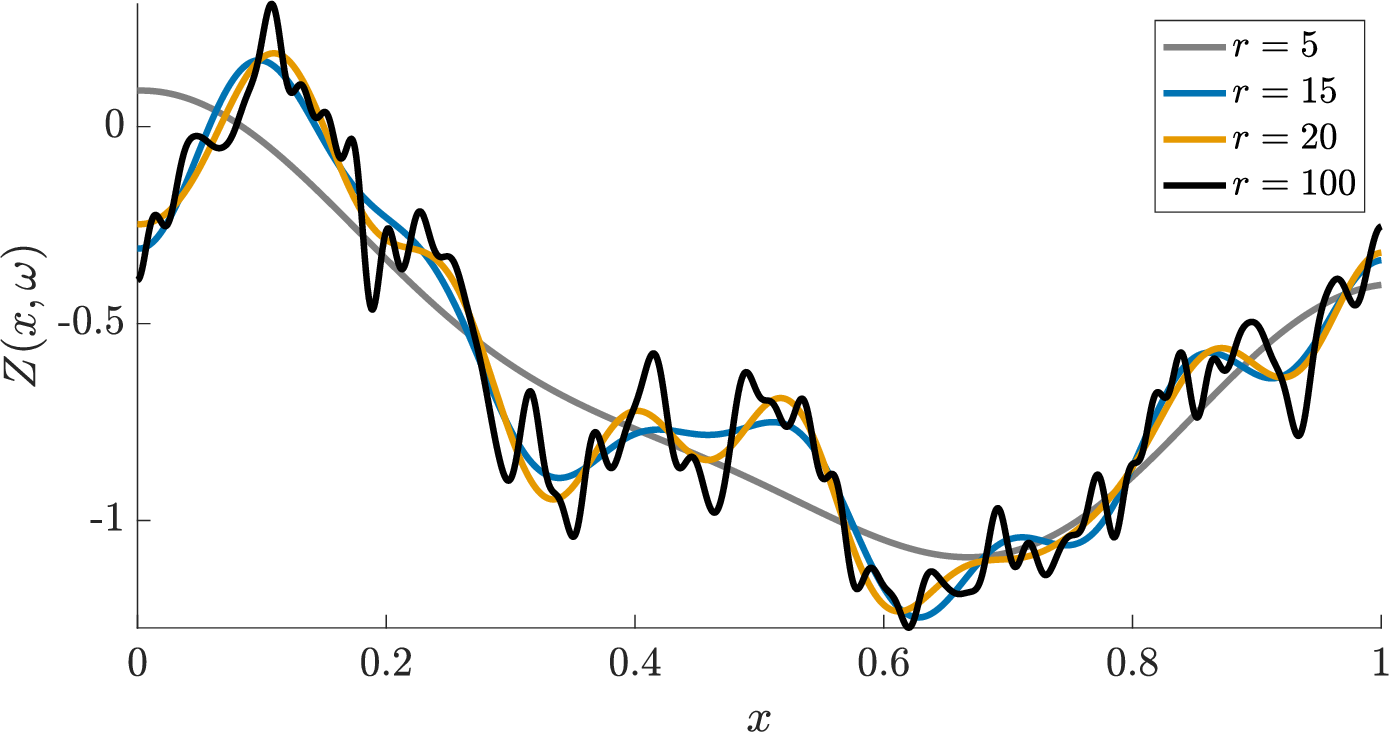}
   \caption{A single realization of $Z(x, \omega)$ with $\ell = 1$ and 
   several truncation levels.}
   \label{fig:truncation_study}
\end{figure}
It is important to note that
the truncation rule based on percentage of captured variance, while well-justified, 
is not always the most effective approach. 
Let us consider, for example, the case of $\ell = 1$. In
Figure~\ref{fig:truncation_study}, we consider a single realization of $Z(x,
\omega)$ for several truncation levels. We observe a
notable difference between the case of $r = 5$ and $r = 15$.  Subsequent
increases in the truncation level add further resolution. However, such
increases should be considered with the final goal of computations in mind. 
Specifically, if $Z(x, \omega)$ is an input to a computational model, one can 
retain only the terms in the KLE that have a material impact on the output of
the model.  We point to~\cite{AlexanderianReeseSmithEtAl19} for numerical
investigations of such issues in the context of models governed by partial
differential equations (PDEs). We revisit such goal-oriented approaches 
to obtaining truncated KLEs in the epilogue.

\section{Epilogue}\label{sec:conc}
There is much to learn about theory of random fields.  The reader is encouraged
to augment their reading of this article by studying standard books on this rich
area of study;  see, e.g.,~\cite{Adler10,AshGardner75,Lifshits12,Stein99}.  A
closely related area is probability theory on function spaces and
function-valued random variables~\cite{DaPrato06,DaPratoZabczyk14}.  Our discussion of
numerical methods and our numerical illustrations provide merely a starting
point. For further details regarding computing with KLEs, see,
e.g.,~\cite{BetzPapaioannouStraub14,Khoromskij09,Eiermann07}.

We end our discussion in this article by considering further perspectives 
on computations with random fields. 
In Section~\ref{sec:gaussian_numerics}, we focused primarily on computing with Gaussian
random fields. However, in practice, it is also common to work with non-Gaussian
fields.
Consider for example a mean-square continuous Gaussian random field
$Z(x, \omega)$ and a field quantity of the form 
\begin{equation}\label{equ:FQ}
Q(x, \omega) = F(Z(x, \omega)), \quad x \in \D, \omega \in \Omega,
\end{equation}
where, as before, $\D$ is a spatial domain and $\Omega$ a sample space.  Suppose
$F$ a nonlinear function taking values in $L^2(\D)$, In this case, $Q$ is not a
Gaussian random field.  

If $F$ in~\eqref{equ:FQ} is sufficiently regular, $Q$ remains a mean
square continuous random field and can be represented with a KLE.
However, in this case, one does not have access to an analytically defined
covariance function, complicating the solution of the eigenvalue
problem~\eqref{equ:eigenvalue_problem}.  In such cases, the covariance operator
of $Q$ may be approximated numerically via sampling, which enables computing an approximate KLE of
$Q$; see, e.g.,~\cite{AlexanderianReeseSmithEtAl19,Smith24}.  

\begin{figure}[ht]\centering
\includegraphics[width=0.7\textwidth]{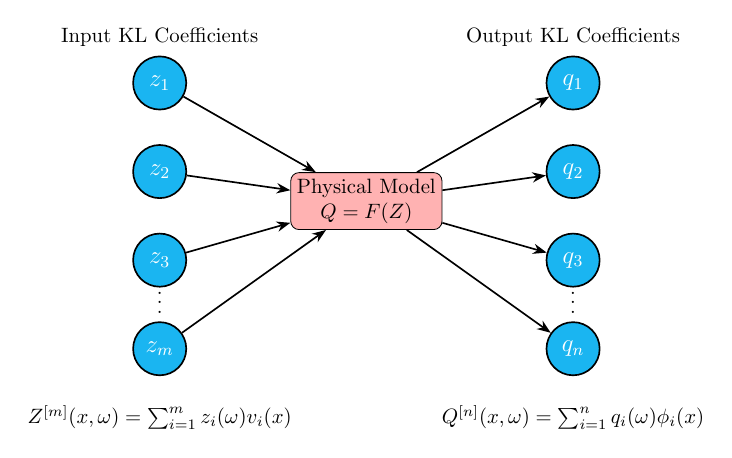}
\caption{Representing input and output of a computational model via the KLE.}
\label{fig:FQ}
\end{figure}

In practice, the input and output fields in~\eqref{equ:FQ} are represented with
truncated KLEs, say with $m$ and $n$ terms, respectively; see
Figure~\ref{fig:FQ}.  Determining $m$ and $n$ can be done following the strategy
outlined in Section~\ref{sec:computational_considerations}. Namely, we can
retain sufficiently many terms in KLE of $Z$ to capture a desired percentage of
the average variance in $Z$.  We may choose a truncation level for the output
field $Q$ in an analogous manner. While this is a principled approach, in
practice, one can perform input and output dimension reduction more efficiently.

In many cases the output field $Q$ is sensitive to only a
few input KL terms. This is the case, for example, when $F$ is defined in terms of the
solution of an elliptic or parabolic PDE where the PDE solution operator may act as
a smoother on the input field $Z$. In such cases, we can retain only the input KL
terms that are most influential to the output field $Q$.  Such influential
KL terms can be determined by performing some form of sensitivity
analysis; see, e.g.,~\cite{CleavesAlexanderianGuyEtAl19,CleavesAlexanderianSaad21}
that rely on global sensitivity analysis to determine important input KL modes. 
Further, when $F$ in~\eqref{equ:FQ} is a PDE-governed mapping, the output
field $Q$ often admits a KLE with rapidly decaying eigenvalues.  This enables
representing the output field $Q$ using only a few KL terms.  

The present discussion indicates that a goal-oriented approach to dimension
reduction is far more effective than using a priori truncation rules based on
percentage of explained variance. In fact,  as argued and demonstrated
in~\cite{ChenArnaudBaptistaEtAl25}, to achieve maximal efficiency, one can
perform optimal input and output dimension reduction simultaneously.  In
general, by fully exploiting problem structure, one can obtain efficient
representations of complex computational models and perform dimensionality
reduction in both input and output spaces.  This is critical for efficient
uncertainty quantification in models with function-valued inputs and outputs.

\bibliographystyle{siamplain}
\bibliography{refs}

\def\cprime{$'$}
\begin{thebibliography}{10}

\bibitem{Adler10}
{\sc R.~J. Adler}, {\em The Geometry of Random Fields}, SIAM, 2010.

\bibitem{Alexanderian15}
{\sc A.~Alexanderian}, {\em A brief note on the {K}arhunen--{L}o\`{e}ve
  expansion}, arXiv:1509.07526,  (2015).

\bibitem{AlexanderianReeseSmithEtAl19}
{\sc A.~Alexanderian, W.~Reese, R.~C. Smith, and M.~Yu}, {\em Model input and
  output dimension reduction using {K}arhunen--{L}o{\`e}ve expansions with
  application to biotransport}, ASCE-ASME Journal of Risk and Uncertainty in
  Engineering Systems, Part B: Mechanical Engineering, 5 (2019), p.~041006,

\bibitem{AshGardner75}
{\sc R.~B. Ash and M.~F. Gardner}, {\em Topics in Stochastic Processes},
  vol.~27 of Probability and Mathematical Statistics, Academic Press, 1975.

\bibitem{BerkoozHolmesLumley93}
{\sc G.~Berkooz, P.~Holmes, and J.~L. Lumley}, {\em The proper orthogonal
  decomposition in the analysis of turbulent flows}, Annual review of fluid
  mechanics, 25 (1993), pp.~539--575.

\bibitem{BetzPapaioannouStraub14}
{\sc W.~Betz, I.~Papaioannou, and D.~Straub}, {\em Numerical methods for the
  discretization of random fields by means of the {K}arhunen--{L}o{\`e}ve
  expansion}, Computer Methods in Applied Mechanics and Engineering, 271
  (2014), pp.~109--129.

\bibitem{Bhatia97}
{\sc R.~Bhatia}, {\em Matrix Analysis}, Springer, 1997.

\bibitem{ChenArnaudBaptistaEtAl25}
{\sc Q.~Chen, E.~Arnaud, R.~Baptista, and O.~Zahm}, {\em Coupled input-output
  dimension reduction: application to goal-oriented {B}ayesian experimental
  design and global sensitivity analysis}, SIAM Journal on Scientific
  Computing, 47 (2025), pp.~A2403--A2430.

\bibitem{CleavesAlexanderianSaad21}
{\sc H.~Cleaves, A.~Alexanderian, and B.~Saad}, {\em Structure exploiting
  methods for fast uncertainty quantification in multiphase flow through
  heterogeneous media}, Computational Geosciences, 25 (2021), pp.~2167--2189.

\bibitem{CleavesAlexanderianGuyEtAl19}
{\sc H.~L. Cleaves, A.~Alexanderian, H.~Guy, R.~C. Smith, and M.~Yu}, {\em
  Derivative-based global sensitivity analysis for models with high-dimensional
  inputs and functional outputs}, SIAM Journal on Scientific Computing, 41
  (2019), pp.~A3524--A3551.

\bibitem{DaPrato06}
{\sc G.~Da~Prato}, {\em An Introduction to Infinite-Dimensional Analysis},
  Springer, 2006.

\bibitem{DaPratoZabczyk14}
{\sc G.~Da~Prato and J.~Zabczyk}, {\em Stochastic Equations in Infinite
  Dimensions}, Cambridge university press, 2014.

\bibitem{Dur98}
{\sc A.~D{\"u}r}, {\em On the optimality of the discrete
  {K}arhunen--{L}o\`{e}ve expansion}, SIAM Journal on Control and Optimization,
  36 (1998), pp.~1937--1939.

\bibitem{Dym13}
{\sc H.~Dym}, {\em Linear Algebra in Action}, vol.~78, American Mathematical
  Soc., 2013.

\bibitem{Eiermann07}
{\sc M.~Eiermann, O.~G. Ernst, and E.~Ullmann}, {\em Computational aspects of
  the stochastic finite element method}, Computing and Visualization in
  Science, 10 (2007), pp.~3--15.

\bibitem{Fukunaga13}
{\sc K.~Fukunaga}, {\em Introduction to Statistical Pattern Recognition},
  Elsevier, 2013.

\bibitem{Ghanem}
{\sc R.~G. Ghanem and P.~D. Spanos}, {\em Stochastic Finite Elements: a
  Spectral Approach}, Springer-Verlag New York, Inc., New York, NY, USA, 1991.

\bibitem{GohbergGoldberg04}
{\sc I.~Gohberg, S.~Goldberg, and M.~A. Kaashoek}, {\em Basic Classes of Linear
  Operators}, 2004.

\bibitem{HirschLacombe99}
{\sc F.~Hirsch and G.~Lacombe}, {\em Elements of Functional Analysis},
  vol.~192, Springer, 1999.

\bibitem{HsingEubank15}
{\sc T.~Hsing and R.~Eubank}, {\em Theoretical Foundations of Functional Data
  Analysis, with an Introduction to Linear Operators}, Wiley Series in
  Probability and Statistics, Wiley, 2015.

\bibitem{Jolliffe02}
{\sc I.~T. Jolliffe}, {\em Principal component analysis}, Springer Series in
  Statistics, Springer-Verlag, New York, second~ed., 2002.

\bibitem{Karhunen47}
{\sc K.~Karhunen}, {\em \"uber lineare methoden in der
  wahrscheinlichkeitsrechnung}, Annales Academiae Scientiarum Fennicae, Series
  A. I: Mathematica-Physica, 37 (1947), pp.~1--79.
\newblock In German.

\bibitem{Keating83}
{\sc J.~P. Keating, J.~E. Michalek, and J.~T. Riley}, {\em A note on the
  optimality of the {K}arhunen--{L}o\`{e}ve expansion}, Pattern Recognition
  Letters, 1 (1983), pp.~203--204.

\bibitem{Khoromskij09}
{\sc B.~N. Khoromskij, A.~Litvinenko, and H.~G. Matthies}, {\em Application of
  hierarchical matrices for computing the {K}arhunen--{L}o{\`e}ve expansion},
  Computing, 84 (2009), pp.~49--67.

\bibitem{Knapp05}
{\sc A.~W. Knapp}, {\em Advanced Real Analysis}, Cornerstones, Birkh\"auser
  Boston, 2005.

\bibitem{Knio10}
{\sc O.~P. Le~Maitre and O.~M. Knio}, {\em Spectral Methods for Uncertainty
  Quantification With Applications to Computational Fluid Dynamics}, Springer,
  2010.

\bibitem{Lifshits12}
{\sc M.~Lifshits}, {\em Lectures on {G}aussian Processes}, Springer, 2012.

\bibitem{LindgrenRue11}
{\sc F.~Lindgren, H.~Rue, and J.~Lindstr{\"o}m}, {\em An explicit link between
  {G}aussian fields and {G}aussian {M}arkov random fields: the stochastic
  partial differential equation approach}, Journal of the Royal Statistical
  Society Series B: Statistical Methodology, 73 (2011), pp.~423--498.

\bibitem{Loeve77}
{\sc M.~Lo{\`e}ve}, {\em Probability Theory {I}}, vol.~45 of Graduate Texts in
  Mathematics, Springer, 1977.

\bibitem{naylorsell}
{\sc A.~W. Naylor and G.~Sell}, {\em Linear Operator Theory in Engineering and
  Science}, Springer-Verlag, New York, 1982.

\bibitem{Rao73}
{\sc C.~R. Rao, C.~R. Rao, M.~Statistiker, C.~R. Rao, and C.~R. Rao}, {\em
  Linear Statistical Inference and its Applications}, vol.~2, Wiley New York,
  1973.

\bibitem{RasmussenWilliams06}
{\sc C.~E. Rasmussen and C.~K.~I. Williams}, {\em Gaussian Processes for
  Machine Learning}, MIT Press, Cambridge, MA, 2006.

\bibitem{ReedSimon72}
{\sc M.~Reed and B.~Simon}, {\em Methods of Modern Mathematical Physics:
  Functional Analysis}, vol.~1, Academic Press, 1972.

\bibitem{Smith24}
{\sc R.~C. Smith}, {\em Uncertainty Quantification: Theory, Implementation, and
  Applications}, SIAM, second~ed., 2024.

\bibitem{Soize17}
{\sc C.~Soize}, {\em Uncertainty quantification}, vol.~23, Springer, 2017.

\bibitem{Stein99}
{\sc M.~L. Stein}, {\em Interpolation of Spatial Data: Some Theory for
  Kriging}, Springer, 1999.

\bibitem{Williams91}
{\sc D.~Williams}, {\em Probability with Martingales}, Cambridge Mathematical
  Textbooks, Cambridge University Press, Cambridge, 1991.

\end{thebibliography}
\bigskip 

\appendix 
\section{Proofs of auxiliary results}
\subsection{Proof of Proposition~\ref{prp:KyFanHilbert}}\label{sec:KyFanProof}
We first introduce some definitions.  Let
$\{u_i\}_{i=1}^r$ be as in the proposition and let $U = \operatorname{span}\{u_1, \ldots, u_r\}$.
We define the orthogonal projection operator of $\H$ onto $U$ as $\PU = \sum_{i=1}^r u_i
\otimes u_i$.  Here, $u_i \otimes u_i$ is the tensor product operator defined by
$(u_i \otimes u_i)x = \ip{u_i}{x}u_i$ for all $x \in \H$. Note also that $\PU$ is a 
self-adjoint operator.
Next, let $\A$ be as in the proposition and define the \emph{restricted} operator 
$\AU = \PU \A \PU$. This is a finite-dimensional operator whose range is a subspace of $U$.
We also note the following useful relation,
\begin{equation}\label{equ:AUuu}
\ip{\AU x}{x} = \ip{\A x}{x}, \quad \text{for all } x \in U.
\end{equation}
Due to the assumptions on $\A$, we have that $\AU$ is a positive self-adjoint 
operator. In particular, $\AU$ can be viewed as a positive self-adjoint operator on $U$. 

We are now ready to prove Proposition~\ref{prp:KyFanHilbert}.  
\emph{Proof}.
Let $U =
\operatorname{span}\{u_1, \ldots, u_r\} \subset \H$ and consider 
$\AU = \PU \A \PU$. 
By the Spectral Theorem, applied to $\AU$, there exists  an orthonormal basis
$\{\phi_j\}_{j=1}^r$ of $U$ given by eigenvectors of $\AU$
with corresponding real non-negative eigenvalues $\{\mu_j\}_{j=1}^r$. We assume these eigenpairs are ordered
so that 
$\mu_1 \geq \mu_2 \geq \cdots \geq \mu_r \geq 0$.  First, note that
\begin{equation}\label{equ:Auu}
\sum_{i=1}^r \ip{\A u_i}{u_i} = \sum_{i=1}^r \ip{\AU u_i}{u_i} 
= \trace(\AU) 
= \sum_{j=1}^r \ip{\AU \phi_j}{\phi_j} 
= \sum_{j=1}^r \mu_j,
\end{equation}
where 
in the second equality, we have used the fact that 
$\{u_i\}_{i=1}^r$ is an orthonormal basis of $U$.
Next, note that by the Courant--Fischer min-max theorem, 
for each $j \in \{1, \ldots, r\}$, the $j$th eigenvalue $\lambda_j$ of $\A$ is characterized by
\[
\lambda_j = \sup_{\substack{W \subset \H \\ \dim W = j}} \inf_{\substack{x \in W \\ \|x\|=1}} \ip{\A x}{x}.
\]
Letting $W_j = \operatorname{span}\{\phi_1, \ldots, \phi_j\} \subset U \subset \H$,
we have
\begin{equation}\label{equ:KF}
\lambda_j \geq 
\inf_{\substack{x \in W_j \\ \|x\|=1}} \ip{\A x}{x} = 
\inf_{\substack{x \in W_j \\ \|x\|=1}} \ip{\AU x}{x} = 
\mu_j.
\end{equation}
Thus, we have 
\[
\lambda_j \geq \mu_j
\quad \text{for all } j \in \{1, \dots, r\}. 
\]
Therefore, using~\eqref{equ:Auu}, we have 
$\sum_{i=1}^r \ip{\A u_i}{u_i} = \sum_{j=1}^r \mu_j \leq \sum_{j=1}^r \lambda_j$.\hfill\proofbox

\bigskip 
\bigskip 

\subsection{Proof of Lemma~\ref{lem:cont}}\label{appdx:proof_cont}
Let $c$ be the covariance function of $Z$.
Suppose $c$ is continuous, and let $\{x_n\}_{n \in \N}$ be a sequence 
that converges to $x \in \D$. Note that 
\begin{multline*}
   \E[(\ZZ{x_n} - \Zx)^2] 
      = \E[\ZZ{x_n}^2]-2\E[\ZZ{x_n}\Zx]+\E[\Zx^2] \\ 
      = c(x_n, x_n) - 2 c(x_n,x) + c(x, x) \to 0,
\end{multline*}
as $n \to \infty$.
Hence, $Z$ is mean-square continuous.   
Conversely, suppose $\Zx$ is mean-square continuous.
Let $\{(x_n, y_n)\}_{n \in \N}$ be a sequence that converges to $(x, y)$. 
Then,
\begin{align*}
|c(x_n, y_n) - c(x, y)| 
    &= |\mathbb{E}[Z(x_n) Z(y_n)] - \mathbb{E}[Z(x) Z(y)]| \\
    &= \big| \mathbb{E}[(Z(x_n)-Z(x))(Z(y_n)-Z(y))] \notag \\
    &\quad + \mathbb{E}[(Z(x_n) - Z(x))Z(y)] + \mathbb{E}[(Z(y_n) - Z(y))Z(x)] \big| \\
    &\leq \big| \mathbb{E}[(Z(x_n)-Z(x))(Z(y_n)-Z(y))] \big| \notag \\
    &\quad + \big| \mathbb{E}[(Z(x_n) - Z(x))Z(y)] \big| + \big| \mathbb{E}[(Z(y_n) - Z(y))Z(x)] \big|. 
\end{align*}
Therefore,
applying the Cauchy--Schwarz inequality to each term on the right-hand side 
and using mean-square continuity of $Z$, we have 
$\lim\limits_{n\to\infty}|c(x_n, y_n) - c(x, y)| = 0$. \hfill\proofbox

\subsection{Proof of Lemma~\ref{lem:Covariance}}\label{appdx:proof_covariance}
Since the process $Z$ is mean-square continuous
Lemma~\ref{lem:cont} implies that its covariance function 
$c(x, y) = \E\{\Zx \Zy\}$ is
continuous. Also, the covariance function is clearly symmetric.
Hence, by Lemma~\ref{lem:compactop}, 
$\C$ is a compact self-adjoint operator. It remains to show that 
$\C$ is positive.
We simply note
\begin{align*}
   \ip{\C u}{u} = \int_\D [\C u](x) u(x) \, \dx  
               &=\int_\D \Big(\int_\D \E[\Zx \Zy] u(y) \, \dy \Big) u(x)\,\dx
               \\
               &=\E\Big[\int_\D \int_\D \Zx \Zy u(x) u(y) \, \dx \,\dy\Big]
               \\
               &=\E\Big[\Big(\int_\D \Zx u(x) \,\dx\Big) 
                       \Big(\int_\D \Zy u(y) \, \dy\Big)\Big]
               \\
               &=\E\Big[ \Big(\int_\D \Zx u(x) \,\dx\Big)^2\Big]\geq 0,
\end{align*}
where we used Fubini's Theorem to interchange integrals.\hfill\proofbox 

\subsection{Proof of Lemma~\ref{lem:coeffs}}\label{appdx:coeffs}
To see the first assertion, we use the Fubini's theorem along 
with the fact that $Z$ is a centered process to note
\begin{multline*}
\E[z_i] = \E\Big[\int_\D \Zx v_i(x) \, \dx\Big] =  
\int_\Omega \int_\D \Zx(\omega) v_i(x) \, \dx \, \P(\domega) \\
= \int_\D \int_\Omega \big(\Zx(\omega) v_i(x) \big)\,\P(\domega) \, \dx = 
\int_\D \E[ \Zx ] v_i(x) \, \dx =  0.
\end{multline*}
To see the second assertion, we proceed as follows
\begin{align*}
\E[z_i z_j] &= \E\Big[ \Big(\int_\D \Zx v_i(x) \, \dx\Big)
                      \Big(\int_\D \Zy v_j(y) \, \dy\Big)\Big]\\
             &= \E\Big[ \int_\D \int_\D \Zx v_i(x) 
                      \Zy v_j(y) \, \dx \, \dy\Big] \\
             &= \int_\D \int_\D \E[\Zx \Zy] v_i(x) v_j(y) \, \dx \, \dy \\
             &= \int_\D \Big(\int_\D c(x,y) v_j(y) \, \dy \Big)v_i(x) \, \dx \\
             &= \int_\D [\C v_j](x) v_i(x) \, \dx 
             = \ip{\C v_j}{v_i} 
             = \lambda_j \delta_{ij},
\end{align*}
where we have used Fubini's Theorem to interchange integrals and
the last conclusion follows from orthonormality of eigenvectors of $\C$. \hfill\proofbox

\end{document}